\documentclass[a4paper,11pt]{article}
\usepackage{latexsym, amssymb, enumerate, amsmath,epsfig}
\usepackage[margin=1in]{geometry}
\usepackage{setspace,color}
\usepackage[titletoc]{appendix}
\usepackage{epstopdf}
\usepackage{graphicx}
\usepackage{pgfplots}
\usepackage{float}
\usepackage{caption}
\usepackage{subcaption}
\usepackage{grffile}
\usepackage{amsthm}
\usepackage[metapost]{mfpic}
\usepackage{verbatim}
\usepackage{lineno}
\usepackage[ruled]{algorithm2e}

\newcommand{\bx}{\mathbf{x}}

\newcommand{\bu}{\mathbf{u}}
\newcommand{\n}{\mathbf{n}}

\newtheorem{defn}{Definition}
\newtheorem{lem}{Lemma}
\newtheorem{thm}{Theorem}

\begin{document}

\title{A Constrained Least-Squares Ghost Sample Points (CLS-GSP) Method for Differential Operators on Point Clouds}

\author{
Ningchen Ying\thanks{Department of Mathematics, the Hong Kong University of Science and Technology, Clear Water Bay, Hong Kong. Email: {\bf nying@connect.ust.hk}}
\and
Kwunlun Chu\thanks{Department of Mathematics, Statistics and Insurance, Hang Seng University of Hong Kong, Hong Kong. Email: {\bf alanchu@hsu.edu.hk}}
\and
Shingyu Leung\thanks{Department of Mathematics, the Hong Kong University of Science and Technology, Clear Water Bay, Hong Kong. Email: {\bf masyleung@ust.hk}}
}
\date{}
\maketitle

\begin{abstract}
We introduce a novel meshless method called the Constrained Least-Squares Ghost Sample Points (CLS-GSP) method for solving partial differential equations on irregular domains or manifolds represented by randomly generated sample points. Our approach involves two key innovations. Firstly, we locally reconstruct the underlying function using a linear combination of radial basis functions centered at a set of carefully chosen \textit{ghost sample points} that are independent of the point cloud samples. Secondly, unlike conventional least-squares methods, which minimize the sum of squared differences from all sample points, we regularize the local reconstruction by imposing a hard constraint to ensure that the least-squares approximation precisely passes through the center. This simple yet effective constraint significantly enhances the diagonal dominance and conditioning of the resulting differential matrix. We provide analytical proofs demonstrating that our method consistently estimates the exact Laplacian. Additionally, we present various numerical examples showcasing the effectiveness of our proposed approach in solving the Laplace/Poisson equation and related eigenvalue problems.
\end{abstract}

\section{Introduction}

Classical numerical methods for partial differential equations (PDEs) usually rely on a structured computational mesh. Finite difference methods (FDM) and finite element methods (FEM), in particular, assume a certain degree of connectivity among sample points and necessitate the computational domain to be properly discretized or triangulated. Due to the connectivity between these sample points, mesh surgery or developing a robust remeshing strategy may be necessary for mesh adaptation, which can be challenging, especially in high dimensions.

Meshless methods \cite{bkofk96,wen05,fas07,nrbd08}, therefore, present a more computationally attractive alternative when we lack control over the allocation of sample points or when obtaining a good discretization of the computational domain is problematic. These meshless methods can be roughly classified into two groups. The first group is the so-called generalized finite difference (GFD) methods \cite{for88,for98}, which extend the traditional FDM based on classical Taylor's expansion. Further studies on this method have been conducted in \cite{benuregav01,benuregav07}. A more popular group follows a similar approach but replaces the standard basis in Taylor's expansion with radial basis functions (RBF) \cite{har71,for96,buh03}. The idea is to locally represent the surface by interpolating the data points using a set of RBFs centered at each sample point. These RBF interpolation methods have demonstrated the ability to achieve spectral accuracy under certain assumptions regarding the form of the bases and the properties of the underlying functions \cite{buhnir93,forfly05}. To expedite computations, the so-called radial basis function finite-differences (RBF-FD) method was proposed by \cite{wanliu02,shudinyeo03,wri03}. The idea is to replace the global interpolation described above with a local interpolation, such that the reconstruction and derivative depend only on values in a small neighborhood around the center of the approximation. The first step of this approach is to gather a subset of sample points around the center using any standard algorithm such as K-nearest neighbors (KNN). Then, the global RBF approach is replaced by this smaller subset of points. 

This paper extends the work developed in \cite{yin17} and introduces a novel meshless method that offers several key advantages. Firstly, they exhibit robustness when given sample points are close to each other; the methods remain stable even if sample points become arbitrarily close or overlap. Secondly, they are local methods, meaning that the approximation to the difference operator depends only on sample points in a small neighborhood of the center. Therefore, the methods are computationally efficient and optimal in computational complexity. Finally, although the methods share some similarities with RBF methods, we must emphasize that our proposed methods differ from conventional RBF-type methods. Specifically, our methods do not interpolate data points but instead approximate the underlying function based on a least-squares approach.

We introduce our proposed approach as the \textit{constrained least-squares ghost sample points} (CLS-GSP) method. It comprises two main components. Initially, we represent a local approximation as a linear combination of a set of RBFs. However, unlike typical RBF interpolation methods (or even RBF-FD methods), where the centers of these RBFs coincide with the sample locations, we suggest replacing these locations with a pre-chosen set of \textit{ghost sample points}. These ghost sample points can be arbitrarily selected if they are well-separated around the center. This effectively mitigates the problem of co-linearity of the basis functions and the ill-conditioning of the coefficient matrix. To enhance the robustness of the algorithm, we propose replacing the interpolation problem with a least-squares problem, minimizing the sum of squared mismatches at all ghost sample points in the local reconstruction to the function value. This substitution allows for greater flexibility in the sampling locations, including the allowance for overlapping given sample points.

It is essential to reiterate again that our method differs from typical RBF interpolation approaches in that our reconstruction does not interpolate all function values but only at the center location where the local approximation is sought. This implies that spectral accuracy cannot be expected in our method. One might argue against replacing interpolation with least-squares fitting, which sacrifices this spectral accuracy. However, spectral accuracy is contingent upon certain conditions being met \cite{forflyrus08,flyleh10}, including the smoothness of the underlying function and the distribution of sample points. Addressing the first concern, one might consider only linear elliptic problems, which imply a smooth solution. However, regarding the second issue, we typically have less control over the distribution of the point cloud and can only assume that the sampling density satisfies certain regularity conditions. Our proposed approach, therefore, offers a simple alternative to approximate derivatives of a function with reasonable accuracy in many applications where these assumptions are not satisfied. 

The second major component of our proposed method emphasizes the contribution of the function value at the center location of each reconstruction. Inspired by the approach developed in a series of studies \cite{leuzha0801,leulowzha11,wanleuzha18}, we observe that increasing the weighting at the center point in the least-squares (LS) reconstruction improves both the accuracy and the diagonal dominance of the differential matrix (DM). In this work, we propose a straightforward idea to enforce a hard constraint: the least-squares reconstruction must pass through the function value at the center location. This concept shares similarities with the local tangential lifting method \cite{wuchiche10}, developed for estimating surface curvature and function derivatives on \textit{triangulated surfaces}. This method has also been recently applied to approximate differential operators and solve PDEs on surfaces \cite{chewu13,chechiwu15,chewu16}, albeit solely for triangulated surfaces, whereas we do not require any connectivity in the sampling points.

This paper is organized as follows. 
In Sections \ref{CLSRBF}, we will discuss our proposed Least-Squares Ghost Sample Points (LS-GSP) and Constrained Least-Squares Ghost Sample Points (CLS-GSP) Methods. This section will provide a more detailed comparison of our approach with various methods. Some theoretical properties of these methods will be provided in Section \ref{Sec:Consistency}. In particular, we will examine its consistency when approximating the standard Laplacian. Some numerical experiments are presented in Section \ref{NumEx} to demonstrate the effectiveness and robustness of our proposed approach. Finally, we provide a summary and conclusion in Section \ref{Summary}.

\section{Our Proposed Methods}
\label{CLSRBF}

We introduce two new ideas for approximating differential operators on randomly distributed sample points. Firstly, in Section \ref{SubSec:LS-GSP}, we introduce the idea of ghost sample points and their use in function reconstruction based on least-squares. Like other LS methods, this approach does not guarantee that the local representation will pass through any data point, particularly the center point. Consequently, the approximation to the difference operator may be inconsistent with the exact value. Therefore, we propose a further enhancement called the constrained least-squares radial basis function method (CLS-GSP) in Section $\ref{SubSec:CLS-GSP}$ to improve the reconstruction. Although we focus mostly on the two-dimensional case in the following discussions, generalizing to higher dimensions is straightforward.

\begin{figure}[!ht]
\centering
\includegraphics[width=0.4\textwidth ]{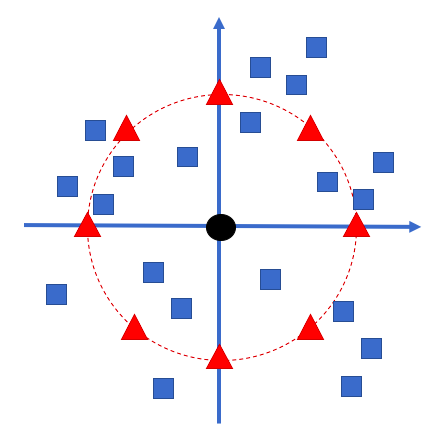}
\caption{The graphical representation of ghost sample points. The black dot represents the center point $\mathbf{x}_0$. Blue squares denote the $n$ nearest scattered data points around $\mathbf{x}_0$. Red triangles correspond to the proposed $d$ ghost sample points associated with the center point $\mathbf{x}_0$. Here, we pick $d=8$ ghost sample points uniformly distributed on a ball with a radius given by the mean distance to the sample points.}
\label{Fig:GhostPoints}
\end{figure}

\subsection{The Least-Squares Ghost Sample Points (LS-GSP) Method} 
\label{SubSec:LS-GSP}

Approximating a given differential operator at $\mathbf{x}=\mathbf{x}_0$ involves two primary steps. The first step entails introducing a set of ghost sample points, followed by a least-squares approximation based on RBF centered at these ghost sample points.

We first gather a subset of sample points ${\mathbf{x}_i,i=1,\ldots,n}$ within a small neighborhood of $\mathbf{x}_0$. Within this neighborhood, we introduce a set of $d$ basis functions, each centered at a predefined location denoted by ${\hat{\mathbf{x}}_m}$ for $m=1,\ldots,d$. These $d$ locations are referred to as the \textit{ghost sample points} associated with the center location $\mathbf{x}=\mathbf{x}_0$. A typical arrangement is illustrated in Figure \ref{Fig:GhostPoints}. The selection of these locations offers considerable flexibility. This work suggests two simple choices for allocating these ghost sample points. The first one is to have ${\hat{\mathbf{x}}_m,m=1,\ldots,8}$ (i.e., $d=8$) evenly distributed on the boundary of a ball centered at $\mathbf{x}_0$ with the radius given by the average distance from the sample points ${\mathbf{x}_i,i=1,\ldots,n}$ to the center. The second one is to have them uniformly located in a small neighborhood centered at the center point. Figure \ref{Fig.Poisson0}(b) shows a distribution with 49 ghost sample points uniformly placed in a disc. This allows these ghost sample points to spread widely and evenly in the local neighborhood. Indeed, one might randomly select a certain fraction of sample points as ghost sample points and determine a stable local approximation of the function using cross-validation. This approach, however, might significantly increase computational time and will not be further investigated here.

The method for selecting these ghost sample points in higher dimensions can also be quite natural. For instance, in three dimensions, we might opt for a sphere of a given radius and uniformly distribute ghost sample points on its surface. Determining these locations can be achieved using Thomson's solution, conceptualizing each ghost sample point as a charged particle that repels others according to Coulomb's law. Explicit coordinates of some configurations are available and can be easily found in the public domain.

Next, we introduce a radial basis function with each ghost sample point and reconstruct the underlying function based on the space spanned by these basis functions. The major difference from the typical RBF approach is that we propose choosing $d \ll n$ so that the typical interpolation problem at ${\mathbf{x}_i,i=1,\ldots,n}$ becomes a least-squares problem. Specifically, we reconstruct the function in the following form:
\begin{equation}
u_c(\mathbf{x}) = \sum_{m=1}^d \lambda_m \phi(\| \mathbf{x}-\hat{\mathbf{x}}_m \|) + \sum_{j=0}^{l} \mu_j P_j(\mathbf{x}),
\label{Eqn:Fitting-LSRBF}
\end{equation}
and aim to determine the unknown coefficients $\lambda_m$ and $\mu_j$ by minimizing the sum of mismatches $\left[u_c(\mathbf{x})-u(\mathbf{x})\right]^2$ at the sample locations $\mathbf{x}_i$, along with the extra polynomial constraints as in RBF-FD. However, since we are considering a least-squares problem after all, our approach might seem to work even if we do not incorporate extra polynomials into the set of basis functions. There are two reasons why we propose to include them. Firstly, this allows our method to approach the standard RBF-FD as $d$ approaches $n$ and $\hat{\mathbf{x}}_m$ approaches $\mathbf{x}_i$. Therefore, our method recovers all the desirable properties, including the stability behavior of RBF-FD. The more significant reason, however, relates to the consistency of the approximation, which will be fully discussed later in Section \ref{Sec:Consistency}.

The choice of the mismatch in $u(\mathbf{x})$ with $u_c(\mathbf{x})$ is clearly not unique. However, emphasizing local information might be more desirable by imposing soft constraints in applications where the samples contain noise. For instance, we can introduce a weight function such as $\kappa(\mathbf{x}_i)=\exp(-\gamma \| \mathbf{x}_i \|)$ to assign fewer weights to those sample points far away from the center point. Then, the corresponding minimization problem becomes  $\sum_{i=1}^n \kappa(\bx_i) \left[ u(\bx_i)-u_c(\bx_i) \right]^2$. However, the current paper will not further investigate this moving least-squares-type approach.

Mathematically, we find the least-squares solution to 
$$
\begin{pmatrix}
\Phi & \textbf{P}\\
\hat{\textbf{P}}^T & \mathbf{0}
\end{pmatrix}
\begin{pmatrix}
\boldsymbol\lambda\\
\boldsymbol\mu
\end{pmatrix}
=
\begin{pmatrix}
\bu \\
\mathbf{0}
\end{pmatrix}
$$
where
\begin{eqnarray*}
& & \Phi = \begin{pmatrix}
\phi\left(\|\bx_1-\hat{\bx}_1\|\right) & \hdots & \phi\left(\|\bx_1-\hat{\bx}_d\|\right) \\
\vdots & \vdots&\vdots \\
\phi\left(\|\bx_n-\hat{\bx}_1\|\right) & \hdots & \phi\left(\|\bx_n-\hat{\bx}_d\|\right)
\end{pmatrix} \, , \, 
\textbf{P}=\begin{pmatrix}
1 & x_1 & y_1 & x_1^2 & x_1y_1 & \hdots\\
\vdots & \vdots & \vdots & \vdots & \vdots \\
1 & x_n & y_n & x_n^2 & x_ny_n & \hdots\\
\end{pmatrix} \, , \\
& & \hat{\textbf{P}}=\begin{pmatrix}
1 & \hat{x}_1 & \hat{y}_1 & \hat{x}_1^2 & \hat{x}_1\hat{y}_1 & \hdots\\
\vdots & \vdots & \vdots & \vdots& \vdots\\
1 & \hat{x}_d & \hat{y}_d & \hat{x}_d^2 & \hat{x}_d\hat{y}_d & \hdots\\
\end{pmatrix}
\, , \, \boldsymbol \lambda=
\begin{pmatrix}
\lambda_1\\
\vdots\\
\lambda_d\\
\end{pmatrix}
\, , \,
\boldsymbol \mu=
\begin{pmatrix}
\mu_0\\
\vdots\\
\mu_k
\end{pmatrix} \, \mbox{ and } \, 
\bu=
\begin{pmatrix}
u(\bx_1)\\
\vdots\\
u(\bx_n)
\end{pmatrix} \, .
\end{eqnarray*}
The size of the matrices $\textbf{P}$ and $\hat{\textbf{P}}$ are given by $n \times ({ l}+1)$ and $d \times ({ l}+1)$, where ${ l}+1$ is the total number of monomials of degree up to $k$ added to the set of basis functions. For example, if we include polynomials in 2 variables ($x$ and $y$) up to degree $k$, we have a total of $l+1=\frac{1}{2}(k+1)(k+2)$ elements of $P_j$. When we have 3 variables ($x$, $y$ and $z$), we have $l+1=\frac{1}{6}(k^3+11k)+k^2+1$ elements of $P_j$ for polynomials up to degree $k$.

We introduce the second set of linear equations $\hat{\textbf{P}}^T \boldsymbol\lambda = \mathbf{0}$ to mimic the conventional RBF-FD method. In particular, if we choose $d=n$ and $\mathbf{x}_i=\hat{\mathbf{x}}_i$, the approach reduces back to the original RBF-FD method with the addition of polynomial basis functions. In this paper, however, we choose $d\ll n$, and the system becomes overdetermined since there are in total $d+{ l}+1$ unknown coefficients, but the system contains $n+{ l}+1$ equations. The least-squares coefficients $\boldsymbol\lambda$ and $\boldsymbol\mu$ can be obtained by
$$
\begin{pmatrix}
\boldsymbol\lambda\\
\boldsymbol\mu
\end{pmatrix}
=
\begin{pmatrix}
\Phi & \textbf{P}\\
\hat{\textbf{P}}^T & \mathbf{0}
\end{pmatrix}^\dagger
\begin{pmatrix}
\bu\\
\mathbf{0}
\end{pmatrix}$$
where $A^\dagger$ is the pseudoinverse of matrix $A$. 

Finally, we consider computing the derivative of the reconstruction. We denote the differential operator by $\mathcal{L}$. Then, applying the operator to the reconstruction and evaluating it at $\bx=\bx_0$, we have $\mathcal{L}u(\bx_0) = \sum_{i=1}^d \lambda_i \mathcal{L}\phi(\| \bx-\hat{\bx}_i \|)\Big\vert_{\bx=\bx_0} +\sum_{j=0}^{{ l}} \mu_j \mathcal{L}P_j(\bx)\Big\vert_{\bx=\bx_0}$ with $\lambda_i$ and $\mu_j$ determined by solving the above least-squares problem. Alternatively, we could rewrite the above procedure and solve the following explicit problem given by $\mathcal{L}u(\bx_0) = \sum_{i=1}^n w_i u(\bx_i)$ where the weights $w_1,{ \ldots}, w_n$ are given by the first $n$ components in 
$$
\begin{pmatrix}
\Phi^T & \hat{\textbf{P}}\\
\textbf{P}^T & \mathbf{0}
\end{pmatrix}^\dagger
\begin{pmatrix}
\mathcal{L}\boldsymbol\phi\vert_{\bx=\bx_0}\\
\mathcal{L}\boldsymbol P\vert_{\bx=\bx_0}\\
\end{pmatrix} \, .
$$

\subsection{The Constrained Least-Squares Ghost Sample Points (CLS-GSP) Method}
\label{SubSec:CLS-GSP}

The LS-GSP method introduced in the above section works reasonably well even if the given sampling locations $\mathbf{x}_i$ are badly distributed. However, we observe that the approximation to any linear operator might be inaccurate. A similar behavior has been observed in \cite{leuzha0801,leulowzha11,wanleuzha18}. The main reason is that the least-squares approach does not necessarily provide enough weighting from the center location. This has motivated the development of the modified virtual grid difference method (MVGD) \cite{wanleuzha18}, in which we explicitly increase the least-squares weighting at the center location.

In this work, we propose to enforce that the least-squares reconstruction has to pass through the function value $(\mathbf{x}_0, u(\mathbf{x}_0))$ exactly. Mathematically, this constraint implies:
$$
u(\bx_0) = \sum_{i=1}^d \lambda_i \phi(\| \bx_0-\hat{\bx}_i \|) +\sum_{j=0}^{l} \mu_j P_j(\bx_0) \, .
$$
Without loss of generality, we assume $\mathbf{x}_0=\mathbf{0}$ for the remaining discussion. Otherwise, we could make a translation $\tilde{\mathbf{x}}=\mathbf{x}-\mathbf{x}_0$ to transform the center point to the origin. The above equation then implies
$u(\mathbf{0}) = \sum_{i=1}^d \lambda_i \phi(\| \hat{\bx}_i \|) +\mu_0$. By subtracting this equation from equation (\ref{Eqn:Fitting-LSRBF}), the CLS-GSP method will fit the function $u$ by:
$$
u(\bx) = u(\textbf{0})+\sum_{i=1}^d \lambda_i \left[\phi\left(\|\bx-\hat{\bx}_i\|\right)-\phi\left(\|\hat{\bx}_i\|\right)\right]+\sum_{j=1}^{ l}\mu_j P_j(\bx) \, .
$$
We emphasize that the set of polynomial bases in CLS-GSP now excludes the constant basis $P_0(\mathbf{x})=1$ corresponding to the index $j=0$. Therefore, when we substitute $\mathbf{x}=\mathbf{0}$ into the right-hand side, the least-squares approximation gives the exact function value $u(\mathbf{0})$ regardless of the values of $\boldsymbol\lambda$ and $\boldsymbol\mu$.

Following a similar practice in the RBF community, we impose additional constraints 
$$\sum_{i=1}^d \lambda_i P_j(\hat{\mathbf{x}}_i) = 0$$ for each $j=1,\ldots,l$ to improve the conditioning of the ultimate least-squares system. Now, we replace the previous set of basis functions by the following modified basis $\psi_i(\mathbf{x})=\phi\left(\|\mathbf{x}-\hat{\mathbf{x}}_i\|\right)-\phi\left(\|\hat{\mathbf{x}}_i\|\right)$, we obtain the coefficients $\boldsymbol\lambda$ and $\boldsymbol\mu$ by finding the least-squares solution of the following overdetermined system of size $(n+l)\times(d+l)$:
$$
\begin{pmatrix}
\boldsymbol\Psi & \textbf{P}\\
\hat{\textbf{P}}^T & \mathbf{0}
\end{pmatrix}
\begin{pmatrix}
\boldsymbol\lambda\\
\boldsymbol\mu
\end{pmatrix}
=
\begin{pmatrix}
\bu\\
\mathbf{0}
\end{pmatrix} 
$$
where we are abusing the notations of $\bu, \textbf{P}$ and $\hat{\textbf{P}}$ given by
\begin{eqnarray*}
& & \boldsymbol\Psi=\begin{pmatrix}
\psi_1(\bx_1) & \psi_2(\bx_1)&\hdots & \psi_d(\bx_1)\\
\psi_1(\bx_2) & \psi_2(\bx_2)&\hdots & \psi_d(\bx_2)\\
\vdots & & \vdots\\
\psi_1(\bx_n) & \psi_2(\bx_n)&\hdots & \psi_d(\bx_n)\\
\end{pmatrix}\quad
\bu=\begin{pmatrix}
u(\bx_1)-u(\mathbf{0})\\
u(\bx_2)-u(\mathbf{0})\\
\vdots\\
u(\bx_n)-u(\mathbf{0})
\end{pmatrix} \, , \\
& & \textbf{P}=
\begin{pmatrix}
P_1(\bx_1) & P_2(\bx_1)&\hdots & P_{ l}(\bx_1)\\
P_1(\bx_2) & P_2(\bx_2)&\hdots & P_{ l}(\bx_2)\\
\vdots & & \vdots\\
P_1(\bx_n) & P_2(\bx_n)&\hdots & P_{ l}(\bx_n)\\
\end{pmatrix}\quad
\text{and}\quad\hat{\textbf{P}}=
\begin{pmatrix}
P_1(\hat{\bx}_1) & P_2(\hat{\bx}_1)&\hdots & P_{ l}(\hat{\bx}_1)\\
P_1(\hat{\bx}_2) & P_2(\hat{\bx}_2)&\hdots & P_{ l}(\hat{\bx}_2)\\
\vdots & & \vdots\\
P_1(\hat{\bx}_d) & P_2(\hat{\bx}_d)&\hdots & P_{ l}(\hat{\bx}_d)\\
\end{pmatrix} \, .
\end{eqnarray*}
The solution can again be found by taking the pseudoinverse of the $(n+{ l})\times(d+{ l})$ matrix.

Finally, applying a differential operator $\mathcal{L}$ to a function defined at $\bx=\mathbf{0}$, we have $\mathcal{L}u(\mathbf{0})\approx \sum_{i=1}^n w_i (u(\bx_i)-u(\mathbf{0}))$ where $w_1,{ \ldots},w_n$ are determined by the first $n$ components in
$$
\begin{pmatrix}
\boldsymbol\Psi^T & \hat{\textbf{P}}\\
\textbf{P}^T & \mathbf{0}
\end{pmatrix}^{\dagger}
\begin{pmatrix}
\mathcal{L}\boldsymbol \psi\vert_{\bx=\mathbf{0}}\\
\mathcal{L}\boldsymbol P\vert_{\bx=\mathbf{0}}\\
\end{pmatrix}
\, \mbox{ where } \,
\mathcal{L}\boldsymbol \psi\vert_{\bx=\mathbf{0}}=
\begin{pmatrix}
\mathcal{L} \psi_1\vert_{\bx=\mathbf{0}}\\
\vdots\\
\mathcal{L} \psi_d\vert_{\bx=\mathbf{0}}
\end{pmatrix}
\, \mbox{and} \,
\mathcal{L}\boldsymbol P\vert_{\bx=\mathbf{0}}=
\begin{pmatrix}
\mathcal{L} P_1\vert_{\bx=\mathbf{0}}\\
\vdots\\
\mathcal{L} P_{ l}\vert_{\bx=\mathbf{0}}
\end{pmatrix} \, .
$$

\subsection{Comparisons with Some Other Approaches}

While our proposed ghost sample points methods incorporate RBF and virtual grids, they differ from typical RBF approaches and our previously developed MVGD method \cite{wanleuzha18}. Here, we now provide a detailed comparison of these approaches.

Conventional RBF algorithms utilize point clouds for \textit{interpolation} of the underlying function. Subsequently, the interpolant undergoes analytical differentiation. In contrast, LS-GSP or CLS-GSP employs \textit{least-squares} to reconstruct the underlying function by specifying the basis functions at \textit{predetermined} locations, which may not align with the data points. Although RBF-FD narrows its focus to a subset of locations within the point cloud for interpolation, the basis functions remain centered at this subset. Consequently, the quality of interpolation could be heavily contingent upon the quality of the provided point cloud sampling and the subsampling step during neighborhood selection.

While MVGD also employs least-squares fitting, it utilizes a polynomial basis for local approximation of the underlying function. Subsequently, finite difference techniques are applied to determine higher-order derivatives of the function, with any center function value substituted by the data value. For instance, assuming $\tilde{u}_i(x)$ represents the local least-squares approximation of the underlying function $u(x)$ at $x_i$, the Laplacian of the function is given by $\left[ \tilde{u}_i(-\Delta x)-2u_i+\tilde{u}_i(\Delta x)\right]/\Delta x^2$, where $\Delta x$ controls the size of the virtual grid, and $u_i$ denotes the given function value at data point $x_i$. There exist multiple distinctions from the method proposed in this work. Firstly, we substitute the polynomial least-squares fitting with an RBF basis centered at well-chosen locations. Consequently, the set of virtual grids (ghost sample points) is not utilized for finite difference calculations; instead, it already plays a role in the least-squares process. Secondly, unlike MVGD, CLS-GSP does not replace the finite difference value with any function value to enhance the diagonal dominance property; rather, we introduce a hard constraint to ensure that the least-squares reconstruction aligns with the function value at the centered location.

\section{A Consistency Analysis}
\label{Sec:Consistency}

To facilitate our discussion, we focus on the standard Laplace operator $\Delta$ due to its significance in numerous applications. While our discussion is limited to two dimensions, extending our approach to higher dimensions is straightforward. The consistency of any other linear differential operator $\mathcal{L}$ can be similarly extended, although we will not delve into those details here.

In this section, we demonstrate the consistency of the CLS-GSP method. Under some regularity conditions, we will analytically show that the CLS-GSP method provides a consistent estimation of the standard Laplacian of a function with order $k-1$, where $k$ is the order of polynomial basis included in the CLS-GSP method.

\subsection{Approximating the Laplace Operator}
Following the discussion above, the Laplace operator is approximated by
\begin{equation}
\Delta u(\mathbf{0})\approx \sum_{i=1}^n w_i (u(\bx_i)-u(\mathbf{0}))
\quad
\text{where}
\quad
\mathbf{w}=\begin{pmatrix}
\boldsymbol\Psi^T & \hat{\textbf{P}}\\
\textbf{P}^T & \mathbf{0}
\end{pmatrix}^{\dagger}
\begin{pmatrix}
\Delta\boldsymbol \psi\vert_{\bx=\mathbf{0}}\\
\Delta \textbf{P}\vert_{\bx=\mathbf{0}}\\
\end{pmatrix} \, .
\label{eqn-laplacian}
\end{equation}

The vector $\mathbf{w}$ has length $n+{ l}$, where $n$ is the number of sample points and ${ l}$ is the number of elements in the polynomial basis up to the $k$-th order in the representation. In approximating the derivative of $u(\mathbf{0})$, on the other hand, we use only the first $n$ elements in $\mathbf{w}$. In particular, if we include only the linear basis ($k=1$ and $l=2$), we have the following system:
$$
\begin{pmatrix}
w_1\\
w_2\\
\vdots\\
w_n\\
w_{n+1}\\
w_{n+2}
\end{pmatrix}
=
\begin{pmatrix}
\psi_1(\bx_1) & \hdots & \psi_1(\bx_n) & \hat{x}_1 & \hat{y}_1\\
\psi_2(\bx_1) & \hdots & \psi_2(\bx_n) & \hat{x}_2 & \hat{y}_2\\
\vdots & & \vdots & \vdots & \vdots\\
\psi_d(\bx_1) & \hdots & \psi_d(\bx_n) & \hat{x}_d & \hat{y}_d\\
x_1 & \hdots & x_n & 0 & 0\\
y_1 & \hdots & y_n & 0 & 0\\
\end{pmatrix}^\dagger
\begin{pmatrix}
\Delta \psi_1\vert_{\bx=\mathbf{0}}\\
\Delta \psi_2\vert_{\bx=\mathbf{0}}\\
\vdots\\
\Delta \psi_d\vert_{\bx=\mathbf{0}}\\
0\\
0\\
\end{pmatrix}
$$
and the following system if we include up to the quadratic basis ($k=2$ and $l=5$), 
$$
\begin{pmatrix}
w_1\\
w_2\\
\vdots\\
w_n\\
w_{n+1}\\
w_{n+2}\\
w_{n+3}\\
w_{n+4}\\
w_{n+5}\\
\end{pmatrix}
=
\begin{pmatrix}
\psi_1(\bx_1) & \hdots & \psi_1(\bx_n) & \hat{x}_1 & \hat{y}_1 & \hat{x}^2_1 & \hat{x}_1 \hat{y}_1 &\hat{y}^2_1\\
\psi_2(\bx_1) & \hdots & \psi_2(\bx_n) & \hat{x}_2 & \hat{y}_2 & \hat{x}^2_1 & \hat{x}_2 \hat{y}_2 &\hat{y}^2_2\\
\vdots & & \vdots & \vdots & \vdots & \vdots & \vdots & \vdots\\
\psi_d(\bx_1) & \hdots & \psi_d(\bx_n) & \hat{x}_d & \hat{y}_d & \hat{x}^2_d & \hat{x}_d \hat{y}_d &\hat{y}^2_d\\
x_1 & \hdots & x_n & 0 & 0 & 0 & 0 &0\\
y_1 & \hdots & y_n & 0 & 0 & 0 & 0 &0\\
x^2_1 & \hdots & x^2_n & 0 & 0 & 0 & 0 &0\\
x_1y_1 & \hdots & x_ny_n & 0 & 0 & 0 & 0 &0\\
y^2_1 & \hdots & y^2_n & 0 & 0 & 0 & 0 &0\\
\end{pmatrix}^\dagger
\begin{pmatrix}
\Delta \psi_1\vert_{\bx=\mathbf{0}}\\
\Delta \psi_2\vert_{\bx=\mathbf{0}}\\
\vdots\\
\Delta \psi_d\vert_{\bx=\mathbf{0}}\\
0\\
0\\
2\\
0\\
2
\end{pmatrix} \, .
$$

\subsection{Consistency}
\label{Def_Order_Consistency}

In the following discussion on consistency, we fix the number of sample points $\mathbf{x}_i=(x_i,y_i)\in\mathbb{R}^2$, where $i=1,\ldots,n$, while shrinking these $n$ points towards the origin by taking $h\mathbf{x}_i$ with $h$ approaching 0. We define $r_i:= h\|\mathbf{x}_i\|_2 = h\sqrt{x_i^2+y_i^2}$. For the ghost sample points chosen evenly distributed on the boundary of a ball, since the radius of this ball equals the average of $r_i$ for $i=1,\ldots,n$, the radius (denoted by $\hat{r}_d$) shrinks to 0 at a rate of $O(h)$. For the second type of distribution where the ghost sample points are chosen evenly within a disc, we can also assume that the radius of the disc is given by the average distance from the sample points to the origin. This also gives the same estimates that $\hat{r}_d=O(h)$.

\begin{defn}
An estimation $\theta_h$ of a quantity $\theta$ is said to be \textit{consistent of order} $k$ if $\vert\theta_h-\theta\vert=O(h^k)$ as $h$ approaches 0.
\end{defn}

The remaining part of this section aims to demonstrate the following consistent property of the approximation to the Laplacian using the CLS-GSP method proposed in the previous section.

\begin{thm}\label{Thm:Consistency}
Assume $u\in C^{k+1}$. Additionally, we assume that the RBF function $\phi$ has the form $\phi(r;c)=f(cr^2)$ with the shape parameter $c$, where $f\in C^2[0,\infty)$. If the sample points $h\mathbf{x}_i$ are taken with $h\to 0$, keeping $ch^2$ constant, and the polynomial basis contains elements up to order $k$, then the approximation (\ref{eqn-laplacian}) is consistent of order $k-1$, i.e.,
$\left\vert \mathbf{w}_A^T \bu - \Delta u(\mathbf{0}) \right\vert=O\left(h^{k-1}\right)$ as $h$ approaches to 0.
\end{thm}

The rest of this section is to prove Theorem \ref{Thm:Consistency}. To begin with, we need the following two lemmas.

\begin{lem}
Denote the first $n$ entries and the remaining ${ l}$ elements of $\mathbf{w}$ in (\ref{eqn-laplacian}) by $\mathbf{w}_A$ and $\mathbf{w}_B$, respectively, so that
$$
\begin{pmatrix}
\mathbf{w}_A\\
\mathbf{w}_B
\end{pmatrix}
=\begin{pmatrix}
\boldsymbol \Psi^T & \hat{\textbf{P}}\\
\textbf{P}^T & \mathbf{0}
\end{pmatrix}^{\dagger}
\begin{pmatrix}
\Delta\boldsymbol \psi\vert_{\bx=\mathbf{0}}\\
\Delta \textbf{P}\vert_{\bx=\mathbf{0}}\\
\end{pmatrix} \, .
$$
If the matrix 
$\begin{pmatrix}
\boldsymbol \Psi & \textbf{P}\\
\hat{\textbf{P}}^T & \mathbf{0}
\end{pmatrix}$ 
is of full rank given by $d+{ l}$, the weights $\mathbf{w}_A$ satisfies $\textbf{P}^T\mathbf{w}_A=\left.\Delta \textbf{P}\right\vert_{\bx=\mathbf{0}}$. 
\label{Lemma A}
\end{lem}

\begin{proof}
Since $\begin{pmatrix} \boldsymbol \Psi & \textbf{P}\\ \hat{\textbf{P}}^T & \mathbf{0} \end{pmatrix}$ is full rank, the transpose $\begin{pmatrix}
\boldsymbol \Psi^T & \hat{\textbf{P}}\\ \textbf{P}^T & \mathbf{0} \end{pmatrix}$ is also full rank. This implies that the solution
$
\begin{pmatrix}
\mathbf{w}_A\\
\mathbf{w}_B
\end{pmatrix}
$
must be a particular solution to the \textit{underdetermined} system given by
\[\begin{pmatrix}
\boldsymbol \Psi^T & \hat{\textbf{P}}\\
\textbf{P}^T & \mathbf{0}
\end{pmatrix}
\begin{pmatrix}
\mathbf{w}_A\\
\mathbf{w}_B
\end{pmatrix}
=
\begin{pmatrix}
\Delta\boldsymbol \psi\vert_{\bx=\mathbf{0}}\\
\Delta \textbf{P}\vert_{\bx=\mathbf{0}}\\
\end{pmatrix}
\]
The second row of the submatrix implies $\textbf{P}^T\mathbf{w}_A=\left.\Delta \textbf{P}\right\vert_{\bx=\mathbf{0}}$. This completes the proof.
\end{proof}

\begin{lem}
Assume the RBF function has the form $\phi(r)=f(cr^2)$ with $f \in C^2[0,\infty)$ where $c$ is the shape parameter. If 
both matrices $\boldsymbol\Psi$ and $\displaystyle \begin{pmatrix}
\boldsymbol \Psi & \textbf{P}\\
\hat{\textbf{P}}^T & \mathbf{0}
\end{pmatrix}$ are full rank, we have
$\left\|\begin{pmatrix}
\mathbf{w}_{A}\\
\mathbf{w}_{B}
\end{pmatrix}\right\|=O\left(h^{-2}\right)$ as $h$ approaches to 0 {  for any vector norm while keeping $ch^2$ as constant.}
\label{Lemma B}
\end{lem}

We found that most of the commonly used RBFs are in the form of $f(cr^2)$ with a twice differentiable $f$. These include the IQ, the MQ, the GA, and the IMQ.

\begin{proof}
We first rewrite the definition of $\mathbf{w}$ and obtain
\begin{equation*}
\begin{pmatrix}
\mathbf{w}_{A}\\
\mathbf{w}_{B}
\end{pmatrix}=\left(
\begin{pmatrix}
\boldsymbol\Psi & \textbf{P}\\
\hat{\textbf{P}}^T & \mathbf{0}
\end{pmatrix}^\dagger\right)^T
\begin{pmatrix}
\left.\Delta \boldsymbol\psi (\bx)\right\vert_{\bx=\mathbf{0}}\\
\left.\Delta \textbf{P} (\bx)\right\vert_{\bx=\mathbf{0}}\\
\end{pmatrix} \, .
\end{equation*}

Therefore, it is sufficient to show the following two conditions:
$
\left\|
\begin{pmatrix}
\boldsymbol\Psi & \textbf{P}\\
\hat{\textbf{P}}^T & \mathbf{0}
\end{pmatrix}^\dagger \right\|
=O(1)$ and $
\left\|
\begin{pmatrix}
\left.\Delta \boldsymbol\psi (\bx)\right\vert_{\bx=\mathbf{0}}\\
\left.\Delta \textbf{P} (\bx)\right\vert_{\bx=\mathbf{0}}\\
\end{pmatrix}\right\|=O(h^{-2})$. For the first condition, we apply the following expansion of the pseudoinverse
$(A+B)^\dagger= A^\dagger - A^\dagger B A^\dagger +O(B^2)$ and take 
$A=\begin{pmatrix}
\Psi & \mathbf{0}\\
\mathbf{0} & \mathbf{0}\\
\end{pmatrix}$ and $B=
\begin{pmatrix}
\mathbf{0} & \textbf{P}\\
\hat{\textbf{P}}^T & \mathbf{0}
\end{pmatrix}$. If the shape parameter $c$ satisfies $ch^2=\text{constant}$ as $h$ goes to 0 (i.e. $c=O(h^{-2})$), we found that $A$ will be independent of $h$ and 
the remaining part $\|B\|$ (without the zeroth order polynomial) is of order $h$. This implies
\[
\left\|\begin{pmatrix}
\boldsymbol \Psi & \textbf{P}\\
\hat{\textbf{P}}^T & \mathbf{0}
\end{pmatrix}^\dagger\right\|=
\left\|
\begin{pmatrix}
\boldsymbol \Psi & \mathbf{0}\\
\mathbf{0} & \mathbf{0}
\end{pmatrix}^\dagger\right\|
+
O(h) \, .
\] 
For the second condition, we first examine $\Delta \boldsymbol\psi(\mathbf{0})$. For each component, we have
\begin{eqnarray*}
\Delta \psi_i(\mathbf{0}) &=& \left.\Delta \phi_i\right\vert_{\bx=\mathbf{0}}=\nabla \cdot \nabla \phi_i(r(x))\vert_{\bx=\mathbf{0}}=\nabla \cdot( \phi_i'(r) \nabla r)\vert_{\bx=\mathbf{0}} \\
&=& \left.\phi_i''(r) \nabla r \cdot\nabla r+\phi_i'(r)\Delta r\right\vert_{\bx=\mathbf{0}} \, .
\end{eqnarray*}

{ 
Now, since $\nabla r \cdot\nabla r =1$ and $\Delta r =(D-1)r^{-1}$ where $D$ is the dimension, we obtain}
$$
\Delta \psi_i(\mathbf{0}) =\left.\phi_i''(r)+ \frac{\phi_i'(r)}{r}\right\vert_{r=\hat{r}_i} = \phi_i''(\hat{r}_i)+ \frac{{ (D-1)}\phi_i'(\hat{r}_i)}{\hat{r}_i}={  \frac{\hat{r}^2_i \phi_i''+(D-1)\hat{r}_i \phi_i'}{\hat{r}^2_i} }\, .
$$
Now, the condition $\phi(r)=f(cr^2)$ with $f\in C^2[0,\infty)$ implies
$r\phi'(r)=2cr^2 f'(cr^2)$ and $r^2\phi''(r)=2cr^2f'(cr^2)+4c^2r^4f''(cr^2)$. These are all functions in the form of $cr^2$. Therefore, if we have a constant $c\hat{r}^2_i$, the quantity $\Delta \psi_i(\mathbf{0})$ can be expressed as a constant divided by $\hat{r}_i^2$. As all the ghost sample points shrink to the origin with speed $h$, the average radius $\hat{r}_i$ is a constant multiple of $h$. This implies that the condition $c\hat{r}^2_i$ being a constant is the same as keeping $ch^2$ constant. Therefore, the quantity $\Delta \psi_i(\mathbf{0})$ can also be expressed as a constant divided by $h^{2}$, and $\vert\Delta \boldsymbol\psi (\mathbf{0})\vert=O(h^{-2})$ when $h$ approaches zero with a constant $ch^2$.

For the term $\left.\Delta \textbf{P} (\bx)\right\vert_{\bx=\mathbf{0}}$, we notice that out of all possible polynomials in the basis, only two polynomials in the basis (given by $x^2$ and $y^2$) whose Laplacian is nonzero. In particular,
$$
\left.\Delta P_j(\bx) \right\vert_{\bx=\mathbf{0}}=\left\{\begin{array}{ll}
2 & \text{if } P_j(\bx)=x_1^2 \text{ or } P_j(\bx)=x_2^2 \text{ or }\ldots \text{ or }  P_j(\bx)=x_D^2\\
0 & \mbox{otherwise.}
\end{array}\right. 
$$
For example, { in the two-dimensional space}, if we take the basis to be $\{x,y,x^2,xy,y^2\}$ respectively, then
\begin{equation*}
\left.\Delta \textbf{P} (\bx)\right\vert_{\bx=\mathbf{0}}=\begin{pmatrix}
\left.\Delta P_1(\bx) \right\vert_{\bx=\mathbf{0}}\\
\left.\Delta P_2(\bx) \right\vert_{\bx=\mathbf{0}}\\
\left.\Delta P_3(\bx) \right\vert_{\bx=\mathbf{0}}\\
\left.\Delta P_4(\bx) \right\vert_{\bx=\mathbf{0}}\\
\left.\Delta P_5(\bx) \right\vert_{\bx=\mathbf{0}}\\
\end{pmatrix}
=\begin{pmatrix}
\left.\Delta x \right\vert_{\bx=\mathbf{0}}\\
\left.\Delta y \right\vert_{\bx=\mathbf{0}}\\
\left.\Delta x^2 \right\vert_{\bx=\mathbf{0}}\\
\left.\Delta xy \right\vert_{\bx=\mathbf{0}}\\
\left.\Delta y^2 \right\vert_{\bx=\mathbf{0}}\\
\end{pmatrix}
=\begin{pmatrix}
0\\
0\\
2\\
0\\
2\\
\end{pmatrix}
\end{equation*}
Combining these two results, we have
\[\left\|
\begin{pmatrix}
\left.\Delta \boldsymbol\psi (\bx)\right\vert_{\bx=\mathbf{0}}\\
\left.\Delta \textbf{P} (\bx)\right\vert_{\bx=\mathbf{0}}\\
\end{pmatrix}\right\|=O(h^{-2}) \, .\]
\end{proof}

Now, we are ready to present the proof of Theorem 1.

\begin{proof}[ {Theorem \ref{Thm:Consistency}}]
We write 
$\mathbf{w}^T_A \bu=w_1(u_1-u_0)+w_2(u_2-u_0)+{ \ldots}+w_n(u_n-u_0)$
where $u_i=u(\bx_i)$ for $i=1,{ \ldots},n$. We first expand all $u_i$ around the origin { up to the $(k+1)$-th order derivatives}. For example, if up to second order polynomial is included $(k=2)$ {  in the case of two dimensions $(D=2)$}, we have
\begin{eqnarray*}
\mathbf{w}_A^T\bu &=& (w_1x_1+\cdots+w_nx_n)u_x + (w_1y_1+\cdots+w_ny_n)u_y \\
&& + \left(w_1x_1y_1+\cdots+w_nx_ny_n \right)u_{xy}\\
&& + \frac{1}{2} \left(w_1x^2_1+\cdots+w_nx^2_n \right) u_{xx} + \frac{1}{2} \left(w_1y^2_1+\cdots+w_ny^2_n\right) u_{yy}\\
&& + \frac{1}{6} \left(w_1x^3_1+\cdots+w_nx^3_n \right) u_{xxx}+ \frac{1}{2} \left(w_1x^2_1y_1+\cdots+w_nx^2_ny_n \right)u_{xxy}\\
&& + \frac{1}{2}\left(w_1x_1y^2_1+\cdots+w_nx_ny^2_n\right) u_{xyy}+\frac{1}{6} \left(w_1y^3_1+\cdots+w_ny^3_n \right) u_{yyy}
\end{eqnarray*}
where the first and second order derivatives $u_x,u_y,u_{xx},u_{xy},u_{yy}$ are evaluated at the origin and the third order derivatives $u_{xxx},u_{xxy},u_{xyy},u_{yyy}$ are evaluated at some point  $\boldsymbol \xi$ inside the circle enclosing all the sample points. This expansion can be easily extended to higher dimensions. Note that the coefficient of all derivatives of $u$ (up to the $k$-th order derivatives) are simply given by components in $\textbf{P}^T\mathbf{w}_A$. Using Lemma \ref{Lemma A} and an intermediate result in Lemma \ref{Lemma B}, we have
$$
\textbf{P}^T\mathbf{w}_A=\left.\Delta \textbf{P}\right\vert_{\bx=\mathbf{0}} =\left\{\begin{array}{ll}
2 & \text{if } P_j(\bx)=x^2 \text{ or } P_j(\bx)=y^2\\
0 & \mbox{otherwise.}
\end{array}\right. 
$$
for $j=1,{ \ldots},l$ up to the highest degree in the polynomial basis. Therefore, if we include up to $k$-th order polynomial basis, then $\mathbf{w}_A^T\bu$ can be expressed as 
\[\mathbf{w}_A^T \bu= u_{xx}+u_{yy}+\sum_{s=0}^{k+1} \sum_{i=1}^n \frac{1}{(k+1)!} \begin{pmatrix}
k+1\\
s
\end{pmatrix} w_i x_i^s y_i^{k+1-s} \frac{\partial^{k+1} u}{\partial x^s\partial y^{k+1-s}} \, . \]
Hence, we have $\left\vert\mathbf{w}_A^T\bu-\Delta u(\mathbf{0})\right\vert$ bounded above by
$$ \frac{2^{k+1}n}{(k+1)!}
\left[ \max_{i=1,\cdots,n}{\vert w_i\vert} \right]
\left[\max_{i=1,\cdots,n}\max(\vert x_i\vert ,\vert y_i\vert)\right]^{k+1} 
{ \left[\max_{s=0,\cdots,k+1}\left\vert\frac{\partial^{k+1} u}{\partial x^s\partial y^{k+1-s}}\right\vert_{\bx=\mathbf{\boldsymbol \xi}}\right]} \, .
$$
{{In general, for higher dimensions, we apply Taylor's expansion for the multivariable function $u$ up the $(k+1)$-th order and cancel out all the lower order terms not in the form of $P_j(x)=x_p^2$. Finally, we obtain the error bounded for $\left\vert\mathbf{w}_A^T\bu-\Delta u(\mathbf{0})\right\vert$} given by
$$
\frac{nD^{k+1}}{(k+1)!}\left[ \max_{i=1,\cdots,n}{\vert w_i\vert} \right]\left[\max_
{\substack{i=1,\cdots, n\\p=1,\cdots, D}} (\vert(\textbf{x}_i)_P\vert)\right]^{k+1}
\left[\max_{\sum \alpha_p = k+1, \alpha_p\geq 0}\left\vert\frac{\partial^{k+1} u}{\partial x_1^{\alpha_1} \ldots \partial x_D^{\alpha_D}}\right\vert_{\bx=\mathbf{\boldsymbol \xi}}\right]
$$ with $(\textbf{x}_i)_p$ being the $p$-th component of the $i$-th sampling point $\textbf{x}_i$.}
Now, based on Lemma \ref{Lemma B}, we have $\max_{i=1,{ \ldots},n}{\vert w_i\vert}=O(h^{-2})$. Moreover, since sample points shrink to the center with speed $h$, we have 
{ 
$\big[\max_{i,p}(\vert(\textbf{x}_i)_P\vert)\big]^{k+1}=O(h^{k+1})$}.
Therefore, we obtain the estimate $\left\vert\mathbf{w}_A^T\bu-\left.\Delta u(\bx)\right\vert_{\bx=\mathbf{0}}\right\vert=O(h^{k-1})$.
\end{proof}

Note that this error bound depends explicitly on the number of data points used in the approximation (i.e., $n$). It might seem that more data would yield less accuracy. However, we are considering the behavior of the approximation for a fixed number of data points. The situation is more complicated when comparing the corresponding behavior for two sets with different numbers of data points. If more data points are collected in the reconstruction, one has to go further away from the center to search for more neighbors, making the approximation less local. Therefore, intuitively, it indeed yields less accuracy. However, if we can control the average radius of these data points from the center as we increase $n$, one may be able to derive a tighter bound for the summation $\sum_{i=1}^n {w_ix_i^s y_i^{k+1-s}}$ independent of $n$. It is, therefore, less conclusive whether the error bounds will actually increase as we increase $n$.

\section{Numerical Examples}
\label{NumEx}

In this section, we will design several numerical examples to demonstrate the effectiveness and performance of the proposed CLS-GSP method in various applications.

\begin{figure}[!ht]
\centering
\includegraphics[width=\textwidth ]{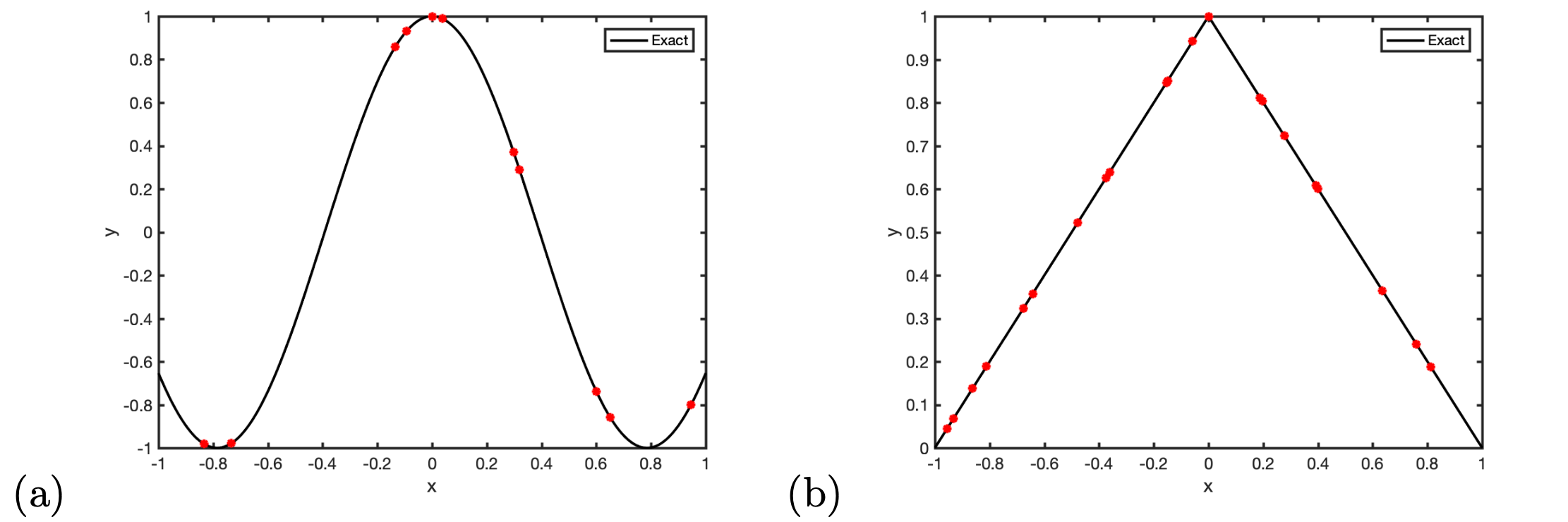}
\caption{(Example \ref{Ex:ErrorApproximation}) A random sampling of the function (a) $u(x) = \cos 4x$ and (b) $u(x) = 1 - \text{sgn}(x)x$. The red circles represent the random sample points chosen in the reconstruction.}
\label{Fig:SmoothErrorApproximation}
\end{figure}

%

\begin{figure}[!ht]
\centering
\includegraphics[width=\textwidth ]{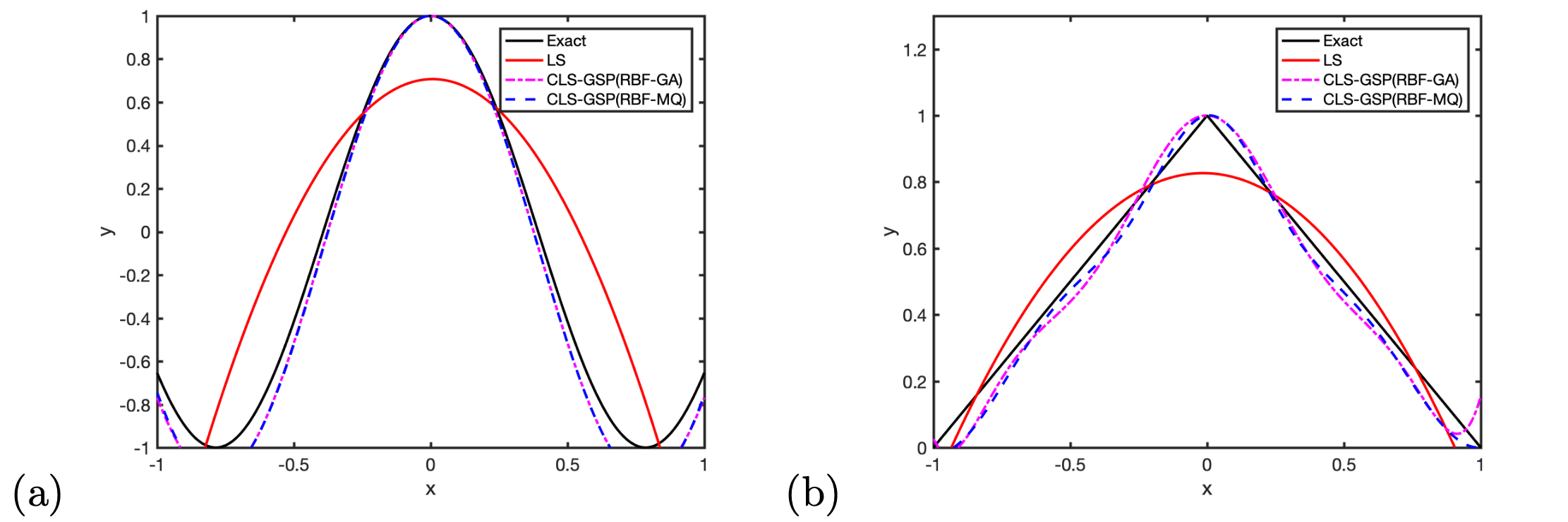}
\caption{(Example \ref{Ex:ErrorApproximation}) Reconstructions of (a) a smooth and (b) a nonsmooth function using our CLS-GSP method with different RBF kernels. The solid black lines are the exact function, and the solid red lines are obtained by the original least-squares reconstruction.}
\label{Fig.Smooth.Singular}
\end{figure}

\subsection{Local Reconstruction}
\label{Ex:ErrorApproximation}

This section demonstrates the performance of the local reconstruction in our proposed CLS-GSP method. We randomly generate sample points between $[-1,1]$ with the center at the origin. In the first example, we consider a smooth function $\cos 4x$, as shown in Figure \ref{Fig:SmoothErrorApproximation}(a), and also a non-smooth function $u(x) = 1 - \text{sgn}(x)x$ which has a kink at the origin, as shown in Figure \ref{Fig:SmoothErrorApproximation}(b). To emphasize the importance of the RBF kernels, we will compare our performance with the standard second-order polynomial least-squares polynomial approach, as shown in Figure \ref{Fig.Smooth.Singular}. The solution from the standard polynomial least-squares method is a little drastic since the method tries to balance the error from \textit{all} sampling locations. The result is especially unsatisfactory since some sample points are far from the center $x_0=0$. Nevertheless, our CLS-GSP method, using either the Gaussian or the Multiquadric basis, can recover the function well, especially in the area near the center $x_0 = 0$, even if the samples are not localized.

\begin{figure}[!ht]
\centering
\includegraphics[width=\textwidth ]{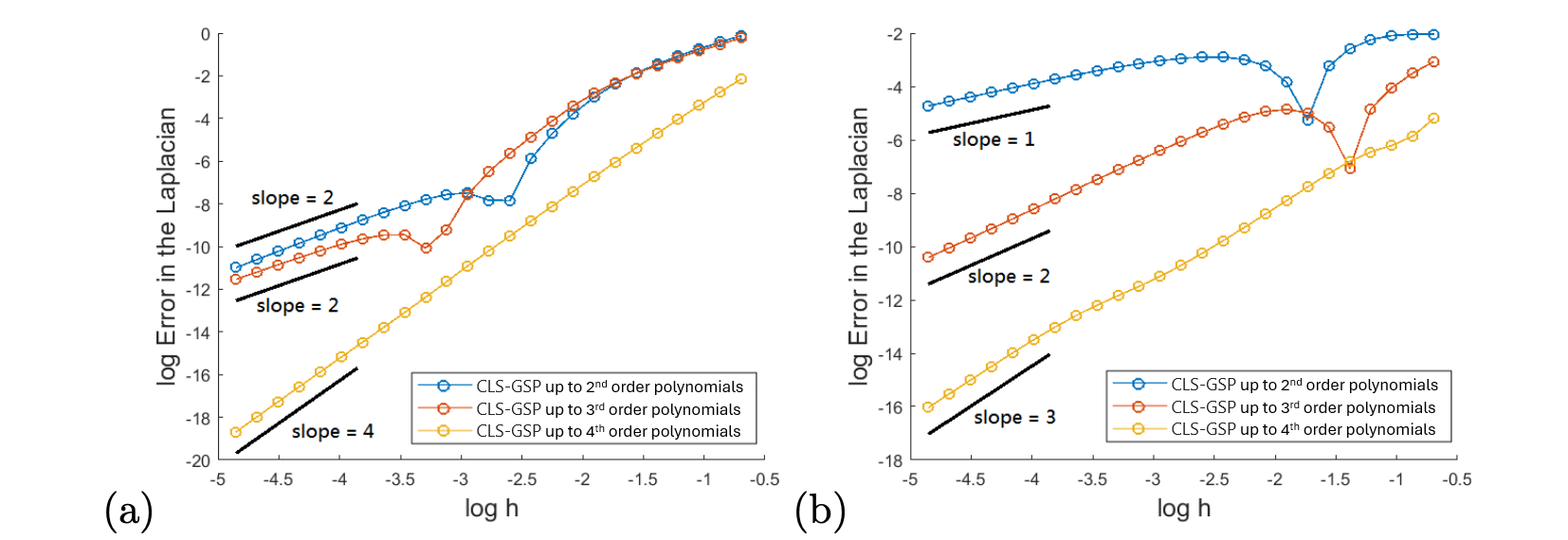}
\caption{{ (Section \ref{Ex:Laplacian_Approximation}) Error in approximating the Laplacian of (a) $u_1(x,y)=\exp(-x^2-y^2)$ and (b) $u_2(x,y)=3\cos x-4\sin y$ at the origin for our CLS-GSP method with the second, the third and the fourth order polynomial basis.}}
\label{Fig.Consistency}
\end{figure}

\subsection{Consistency in Approximating the Laplacian}
\label{Ex:Laplacian_Approximation}
In the second example, we demonstrate the consistency of our CLS-GSP method applied to the Laplace operator, $\Delta$, in $\mathbb{R}^2$. We examine the Laplacian of two functions: $u_1(x,y)=\exp(-x^2-y^2)$ and $u_2(x,y)=3\cos x-4 \sin y$, at $(x,y)=(0,0)$ where the exact Laplacian is given by $-4$ and $-3$, respectively. We generate $20$ random sample points around the center $(0,0)$ and take 8 ghost sample points uniformly on a circle centered at the origin given by $\left\{\left(r\cos \frac{k\pi}{4} ,r\sin \frac{k\pi}{4}\right)\right\}_{k=1}^8$, where $r$ is chosen to be the average distance of the sample points from the origin. We apply our CLS-GSP method using Gaussian kernels (GA) and polynomial basis up to the second, third, and fourth order. The approximation to the Laplacian is calculated using expression (\ref{eqn-laplacian}).

In the first test function $u_1(\mathbf{x})=\exp(-x^2-y^2)$, we refine the scaling parameter $h$ as defined in Section \ref{Def_Order_Consistency} while fixing $ch^2=0.1$ where $c$ is the shape parameter. The log-Error plot is shown in Figure \ref{Fig.Consistency}(a). According to Theorem \ref{Thm:Consistency}, we expect our computed Laplacian to have $(k-1)$-th order consistency when up to $k$-th order polynomial bases are included. In Figure \ref{Fig.Consistency}(a), when $h$ is small, the slopes of the log-Error plot are approximately $2$, $2$, and $4$ when $k=2,3$, and $4$, respectively. This indicates an extra consistent order when up to the second- or fourth-order polynomial bases are included. All the third-order and fifth-order partial derivatives of $u_1(x,y)$ at the origin are zero. This leads to the vanishing of the leading-order term. However, considering the second test function $u_2(\mathbf{x})=3\cos x-4 \sin y$, the consistent behavior is different from $u_1$. Figure \ref{Fig.Consistency}(b) shows the log-Error for different $h$, when $ch^2$ is fixed to be $1$. The slopes are approximately $1$, $2$, and $3$ when $k=2,3$, and $4$, respectively. This observation aligns with Theorem \ref{Thm:Consistency} as some high-order derivatives of $u_2(x,y)$ are non-zero.


\begin{figure}[!ht]
\centering
\includegraphics[width=\textwidth ]{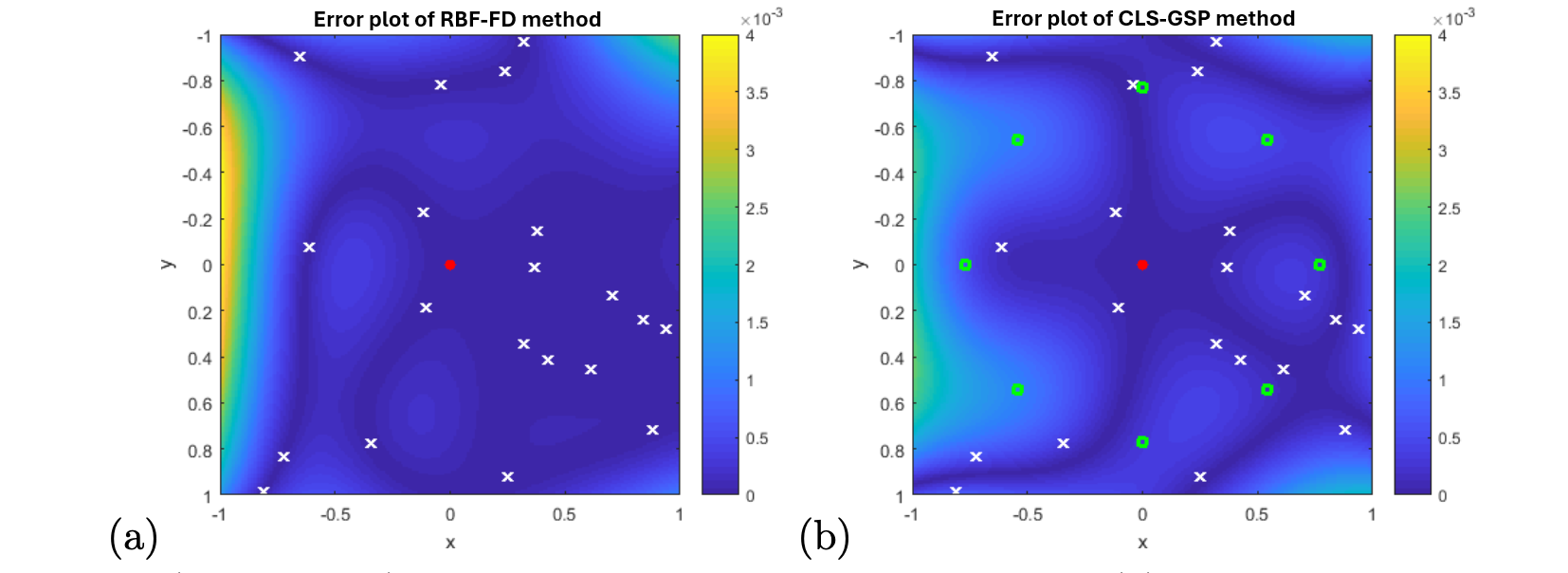}
\caption{(Section \ref{Ex: Shape_Para}) Error in the local reconstruction for (a) the RBF-FD method and (b) the CLS-GSP method using IQ-RBF. The white crosses and the green circle present the location of the sample and ghost points, respectively. The red dot is the origin.}
\label{Fig.Shape1}
\end{figure}

\begin{table}[!ht]
\captionsetup{singlelinecheck=off}
\centering
\begin{tabular}{ |c|cc|cc| } 
\hline
$c\bar{r}^2$ & Error & Rank & Error & Rank
\\
 & in RBF-FD & & in CLS-GSP &\\
\hline
$10^{-8}$ & $4.7959\times 10^{-3}$ & $12$ & $3.9414\times 10^{-3}$ & $10$\\
$10^{-7}$ & $4.7959\times 10^{-3}$ & $12$ & $3.9414\times 10^{-3}$  & $10$\\
$10^{-6}$ & $4.7959\times 10^{-3}$ & $12$ & $3.9414\times 10^{-3}$  & $10$\\
$10^{-5}$ & $4.6101\times 10^{-3}$ & $12$ & $2.1457\times 10^{-3}$  & $10$\\
$10^{-4}$ & $1.3323\times 10^{-4}$ & $12$ & $2.3067\times 10^{-5}$  & $10$\\
$10^{-3}$ & $1.5363\times 10^{-5}$ & $17$ & $4.6044\times 10^{-6}$  & $12$\\
$10^{-2}$ & $4.6426\times 10^{-5}$ & $24$ & $4.1861\times 10^{-5}$  & $13$\\
$10^{-1}$ & $2.5375\times 10^{-4}$ & $26$ & $5.1265\times 10^{-4}$  & $13$\\
$10^0$ & $3.5709\times 10^{-3}$ & $26$ & $1.8594\times 10^{-3}$  & $13$\\
$10^1$ & $4.3898\times 10^{-3}$ & $26$ & $3.0790\times 10^{-3}$  & $13$\\
$10^2$ & $6.5927\times 10^{-3}$ & $26$ & $3.5674\times 10^{-3}$  & $13$\\
$10^3$ & $5.0129\times 10^{-3}$ & $26$ & $3.6423\times 10^{-3}$ & $13$\\
$10^4$ & $4.8179\times 10^{-3}$ & $26$ & $3.6509\times 10^{-3}$  & $13$\\
$10^5$ & $4.7981\times 10^{-3}$ & $26$ & $3.6518\times 10^{-3}$ & $13$\\
\hline
\end{tabular}
\caption{(Section \ref{Ex: Shape_Para}) The error in the numerical Laplacian and the rank of $
\left(\protect\begin{smallmatrix}
\boldsymbol{\Phi} & \textbf{P}\\
\textbf{P}^T & \textbf{0}\\
\protect\end{smallmatrix} \right)
$ in the RBF-FD method and the rank of $\left(\protect\begin{smallmatrix}
\boldsymbol{\Psi} & \textbf{P}\\
\hat{\textbf{P}}^T & \textbf{0}\\
\protect\end{smallmatrix}\right)
$ in the CLS-GSP method with different shape parameters.}
\label{Table.Shape}
\end{table}

\begin{figure}[!ht]
\centering
\includegraphics[width=\textwidth ]{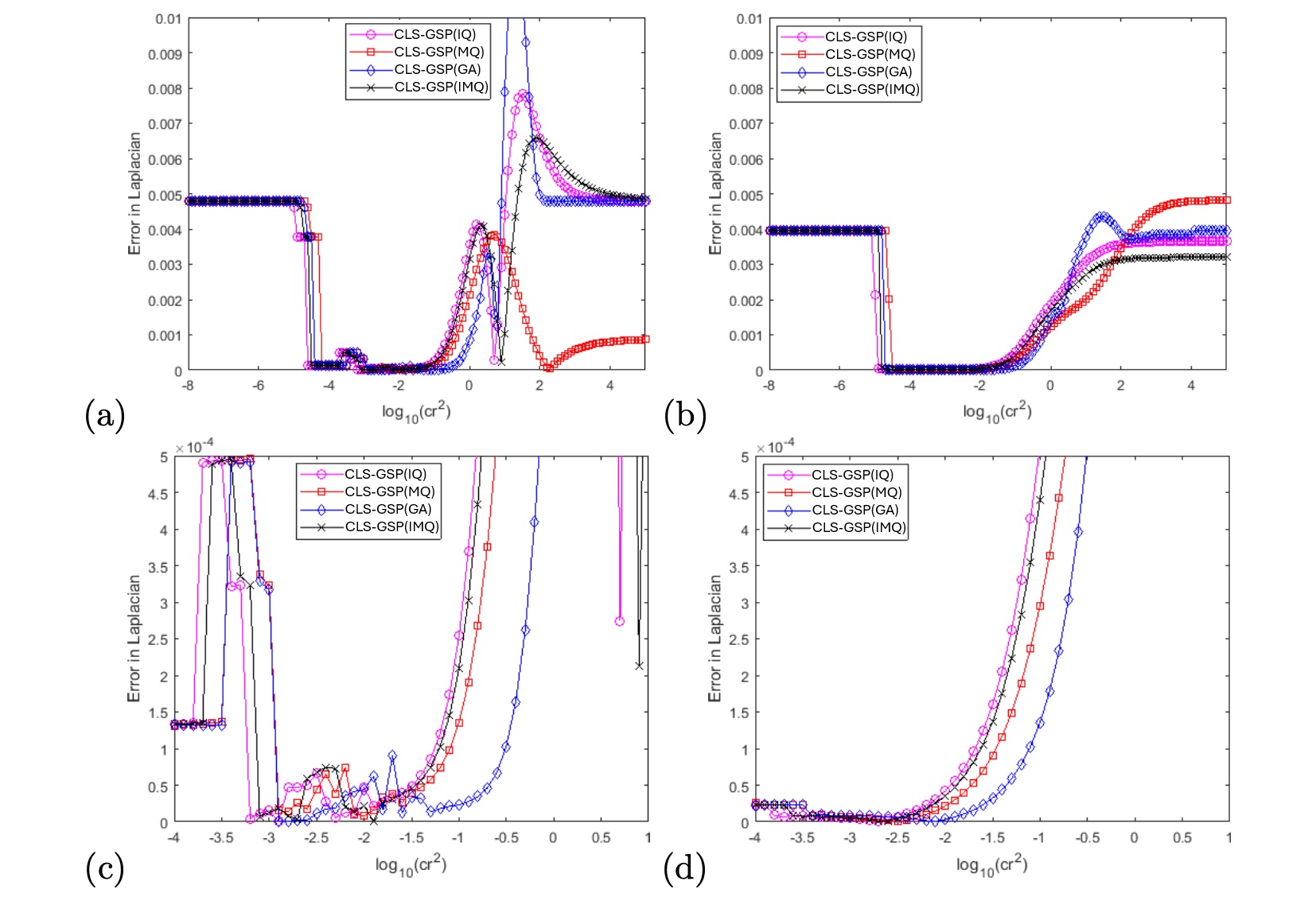}
\caption{(Section \ref{Ex: Shape_Para}) The error in the numerical Laplacian versus the shape parameter using RBF-FD and CLS-GSP methods with 4 common types of basis functions. (a) The RBF-FD with{  $-8\leq \log (c\bar{r}^2)\leq 5$, (b) the CLS-GSP with $-8\leq \log(c\bar{r}^2)\leq 5$, (c) the RBF-FD with $-4 \leq \log(c\bar{r}^2)\leq -1$ and (d) the CLS-GSP with $-4\leq \log(c\bar{r}^2)\leq -1$.}}
\label{Fig.Shape2}
\end{figure}

\begin{figure}[!ht]
\centering
\includegraphics[width=\textwidth ]{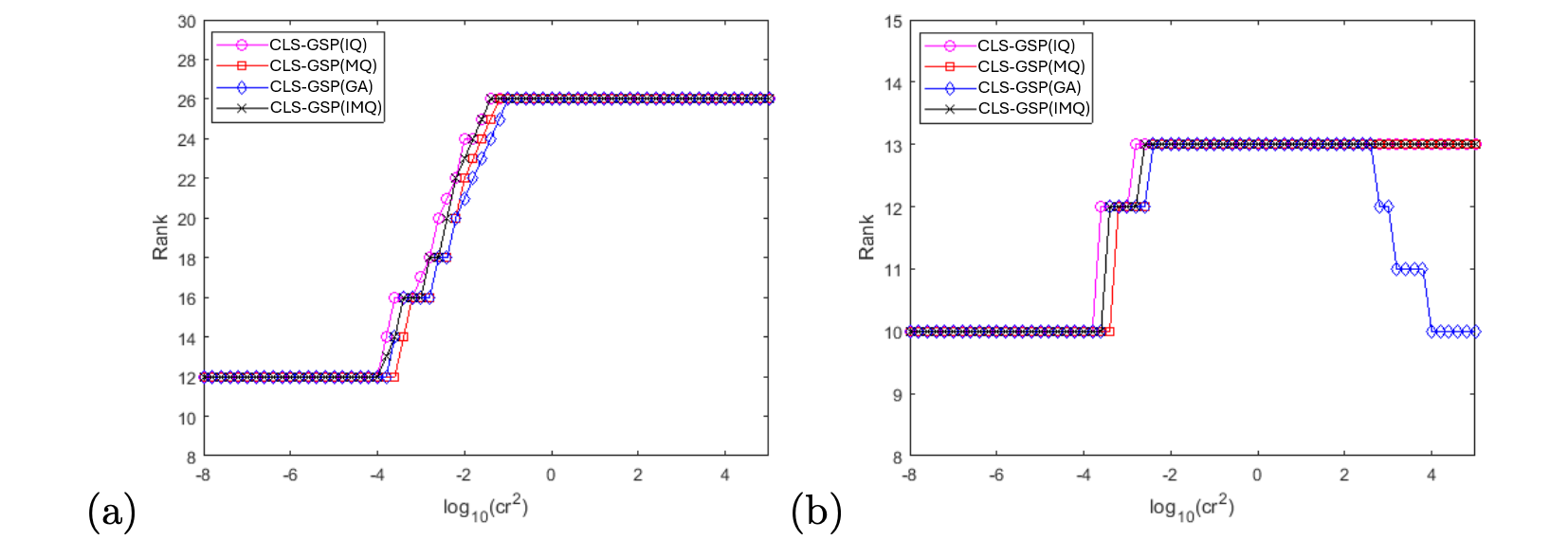}
\caption{(Section \ref{Ex: Shape_Para}) The rank of the corresponding matrices in (a) the RBF-FD and (b) the CLS-GSP method versus the shape parameter with 4 common types of basis functions. 
}
\label{Fig.Shape3}
\end{figure}

\subsection{Sensitivity of the Shape Parameter}
\label{Ex: Shape_Para}

In this section, we test the sensitivity of the shape parameter in both the RBF-FD and CLS-GSP methods. The condition for a certain matrix to be full rank, as required in Lemma \ref{Lemma A} and Theorem \ref{Thm:Consistency}, depends on the choice of the shape parameter in the basis function. In this example, we generate $20$ sample points in $[-1,1]^2$ and examine the performance of both the RBF-FD method and the CLS-GSP method with different types of basis functions and different values of the shape parameter $c$. The underlying function is given by $u(x,y)=\exp{(-0.25x)}\cos(0.2y)$. The exact Laplacian at the origin is $0.0225$. We include the polynomial basis up to the second order in both methods. The ghost points in our proposed CLS-GSP are chosen to be $\left\{\left(\bar{r}\cos \frac{k\pi}{4},\bar{r}\sin \frac{k\pi}{4}\right)\right\}_{k=1}^8$, where $\bar{r}=0.7682$ is the average distance of the $20$ sample points from the origin.

Figure \ref{Fig.Shape1} illustrates the error in the approximation obtained by both the conventional RBF and our CLS-GSP approach. Remarkably, with only 8 basis functions in our CLS-GSP approach, we achieve an approximation of comparable accuracy to the typical RBF approach, which requires 20 basis elements. To test the sensitivity, we maintain the same set of sampling points while varying the parameter $c$. Given that all four commonly used RBFs listed are functions of $cr^2$, it is natural to select $c$ such that $c\bar{r}^2$ remains constant. In our test case, we fix $c$ such that $c\bar{r}^2=10^k$, where $-8\leq k \leq 5$. To determine the optimal choice of the shape parameter, we analyze the error in the numerical Laplacian at the origin and monitor the rank of the matrices
$$
\begin{pmatrix}
\boldsymbol{\Phi} & \textbf{P}\\
\textbf{P}^T & \textbf{0}\\
\end{pmatrix}
\, \mbox{ and } \,
\begin{pmatrix}
\boldsymbol{\Psi} & \textbf{P}\\
\hat{\textbf{P}}^T & \textbf{0}\\
\end{pmatrix} 
$$
in the RBF-FD method and the CLS-GSP method, respectively. We compute the rank numerically by counting the number of singular values greater than the tolerance $\epsilon=10^{-10}$. The results are shown in Table \ref{Table.Shape} for the IQ-RBF case. Initially, we observe that both methods yield large errors when $c$ is extremely large. Conversely, if the shape parameter $c$ is extremely small, both methods encounter a rank-deficiency problem, resulting in a large error in the numerical Laplacian.

We repeated the same experiments to determine a suitable range of shape parameter $c$ but replaced the IQ-type basis with the MQ-, GA-, and IMQ-types, respectively. Figure \ref{Fig.Shape2}(a-b) illustrates the errors in the numerical Laplacian as we vary the quantity $\log(c\bar{r}^2)$. Our proposed CLS-GSP method generally approximates the Laplacian more accurately than the RBF-FD method. Further zooming in on the region $c\bar{r}^2\in[10^{-4},10^{-1}]$, as shown in Figure \ref{Fig.Shape2}(c-d), reveals that the fluctuation in the error in our proposed CLS-GSP method is smaller than that obtained by the RBF-FD method. This indicates that the CLS-GSP method is less sensitive to small changes in the shape parameter $c$.

Figure \ref{Fig.Shape3} illustrates the rank of the corresponding matrices in both methods versus the choice of the shape parameter $c$. Across all types of basis functions, the RBF-FD method exhibits a rank deficiency problem (with full rank being $26$) when $\log_{10}\left(c\bar{r}^2\right)$ drops below roughly $-1.12$. In contrast, our proposed CLS-GSP method (specifically the GA-type) encounters a rank deficiency problem (with full rank being $13$) only when $\log_{10}\left(c\bar{r}^2\right)$ drops below approximately $-2.5$. As discussed in the analysis in Section \ref{Sec:Consistency}, achieving a full-rank matrix is crucial for obtaining the consistency result. Hence, our proposed CLS-GSP method offers a broader range of the shape parameter $c$ to ensure the necessary consistency in the numerical approximation of the required derivative.

Regarding the behavior of the GA-type basis functions, it is evident that the rank of the matrix  $\begin{pmatrix}
\boldsymbol{\Psi} & \textbf{P} \\
\hat{\textbf{P}}^T & \textbf{0} 
\end{pmatrix}$ decreases as we increase the parameter $c\bar{r}^2$ and, consequently, the shape parameter $c$. This phenomenon arises because the support of the Gaussian basis functions becomes excessively narrow with larger values of $c$. Specifically, when $\log_{10}\left(c\bar{r}^2\right)$ is around 2, the corresponding shape parameter $c$ is approximately 170. This indicates that the standard deviation of each Gaussian basis function is roughly 0.05. Comparing this scale to the domain size, which is $O(1)$, it becomes clear that the GA-type basis functions struggle to provide an accurate approximation to the underlying function.

\subsection{Two Dimensional Poisson Equations}
\label{Ex:Poisson}

In this example, we apply the CLS-GSP method to solve Poisson's equation in different 2D domains, defined by the equation:
$$
\left\{\begin{matrix}
\Delta u = f(x,y), & \text{ in } \Omega \\ 
u = g(x,y), & \text{ on } \partial \Omega \, .
\end{matrix}\right.
$$
To mitigate one-sided stencils near the boundary $\partial \Omega$, we incorporate a layer of boundary nodes outside the domain $\Omega$. In our approach, we presume that the function $g(x,y)$ is specified not only on the boundary $\partial \Omega$ but also analytically defined in the exterior, ensuring precise values when necessary. If $g$ is solely provided by an explicit expression on $\partial \Omega$, we determine the function value on these boundary nodes through a normal extension, ensuring $\mathbf{n}\cdot \nabla g=0$, where $\mathbf{n}$ represents the outward normal along $\partial \Omega$. This is achieved by projecting the boundary nodes onto the boundary. Furthermore, if data points are available on the boundary, the value assigned to a boundary node can be determined through additional interpolation on $\partial \Omega$.

Let $N$ denote the number of meshless sample points in the computational domain, and $N_b$ represent the number of boundary nodes outside the domain. Consequently, we obtain a DM of the Laplacian operator given by: $L_g = \begin{pmatrix}
L & L_1
\end{pmatrix}$, where $L$ is an $N \times N$ matrix associated solely with the sample points within the computational domain, and $L_1$ is an $N \times N_b$ matrix containing the coefficients associated with the boundary nodes in $\Omega_b$. This operator operates on all $N+N_b$ points and provides the numerical Laplacian at the $N$ sample points within the computational domain. Set 
\begin{eqnarray*}
\mathbf{f} &=& \left[ f(x_1,y_1), f(x_2,y_2), \cdots, f(x_N,y_N) \right]^T \nonumber\\
\mathbf{g} &=& \left[ g(\hat{x}_1,\hat{y}_1), g(\hat{x}_2,\hat{y}_2), \cdots, g(\hat{x}_{N_b},\hat{y}_{N_b}) \right]^T \, , \nonumber
\end{eqnarray*}
where $(x_i,y_i), i = 1,{ \ldots}, N$ are sample points in $\Omega$ and $(\hat{x}_i,\hat{y}_i), i = 1, { \ldots}, N_b$ are boundary nodes in $\Omega_b$. Therefore, the solution $u$ at the interior sample points, denoted by $\mathbf{u}$, can be computed using $L \mathbf{u} + L_1 \mathbf{g} = \mathbf{f}$, which implies that we solve $L \mathbf{u} = \mathbf{f} - L_1 \mathbf{g}$.

\begin{figure}[!ht]
\centering
\includegraphics[width=\textwidth ]{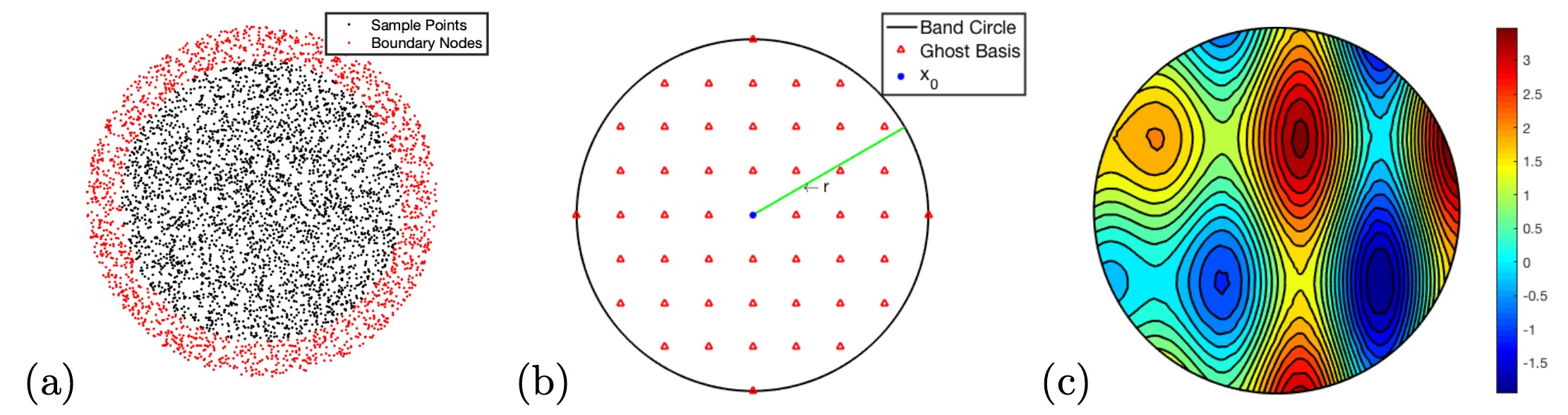}
\caption{(Example \ref{Ex:Poisson} on a disc) (a) The sample points. (b) The allocation of 49 ghost sample points. (c) The numerical solution to the Poisson equation on a disc.}
\label{Fig.Poisson0}
\end{figure}
\begin{figure}[!ht]
\centering
\includegraphics[width=\textwidth]{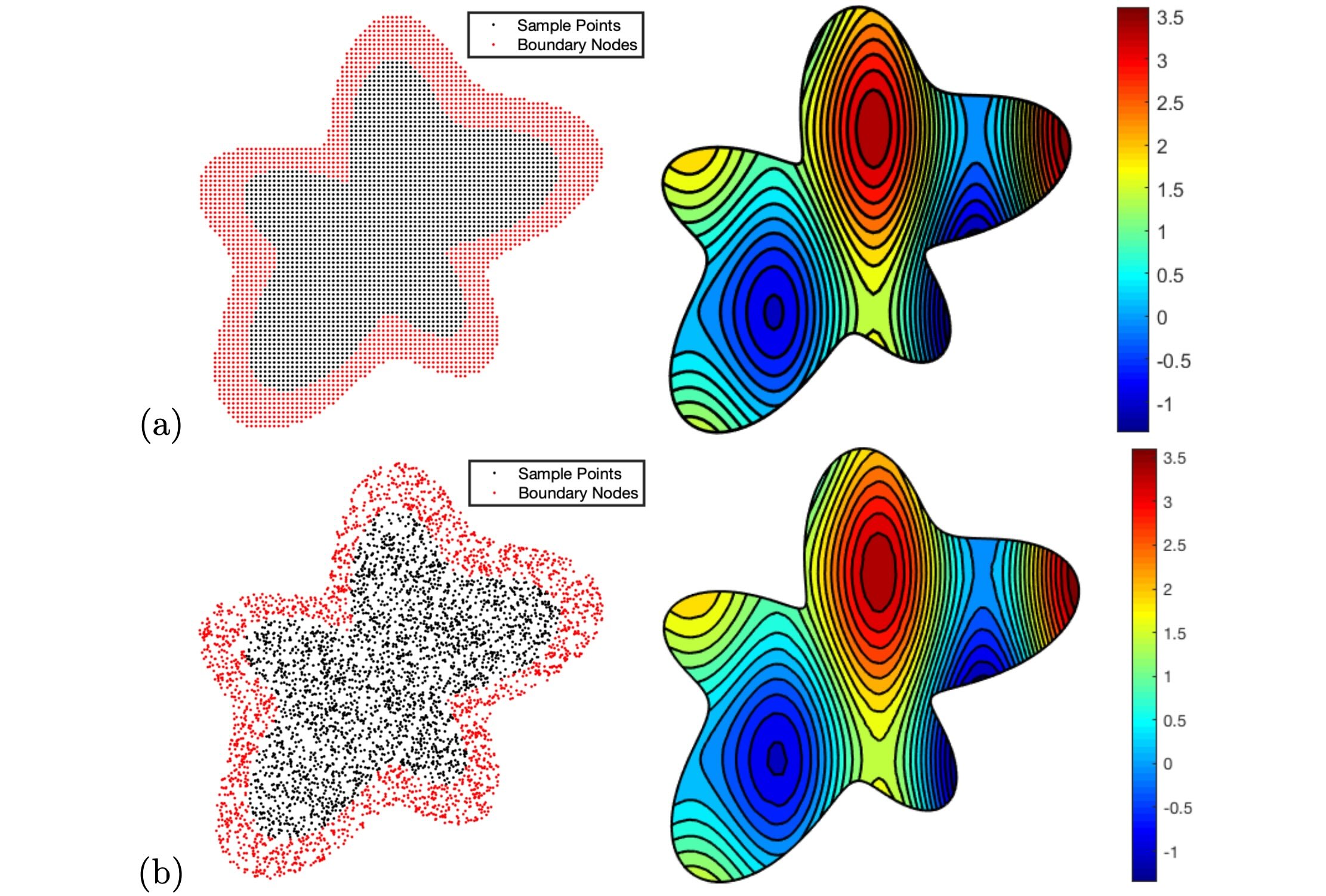}
\caption{{ (Example \ref{Ex:Poisson}: the numerical solution to the Poisson equation on an irregular domain) (a) Uniform sample points and (b) random sample points.}}
\label{Fig.Poisson2}
\end{figure}

In the first domain, we consider a simple disc $\Omega = { (x,y): x^2 + y^2 \leq 2^2}$. Both functions $f(x,y)$ and $g(x,y)$ are determined from the exact solution given by $u(x,y) = 1 + \sin(4x) + \cos(3x) + \sin(2y)$. We select $N=3159$ random points from $x^2 + y^2 \leq 2^2$ and $N_b = 1841$ random boundary nides from $2^2 < x^2 + y^2 \leq 2.5^2$ in the boundary domain $\Omega_b$. These points are plotted in Figure \ref{Fig.Poisson0}(a). Some sample points are located very close, which could cause instability problems with a straightforward implementation of the RBF-FD method based on the $K$-nearest neighbors. In this example, we choose $n=60$ (i.e., we have 60 local neighbors for each sample point $\mathbf{x}_0$) to generate the weights vector of the Laplacian matrix at $\mathbf{x}_0$. We choose 49 ghost basis samples uniformly within a distance of $r = \frac{1}{2} \max_i d(\mathbf{x}_0,\mathbf{x}_i)$, as shown in Figure \ref{Fig.Poisson0}(b). The numerical solution is presented in Figure \ref{Fig.Poisson0}(c), and the $L_\infty$ error in the solution is around $4\times 10^{-3}$.

In the second test domain, we consider an irregular flower shape given by $\Omega = \{ (\theta, r) : r\leq r_{\max}(\theta)=1.4+0.4\sin(5\theta)+0.4\sin(2\theta), \theta \in [0, 2\pi) \}$. The functions $f(x,y)$ and $g(x,y)$ are again determined from the same exact solution. In this example, we test on two different sets of sample points. In the first sampling, we select 2677 samples uniformly inside the computational domain with 2066 boundary nodes chosen in $\Omega_b = \{ (\theta, r) : r_{\max}(\theta) < r\leq r_{\max}(\theta)+0.5, \theta \in [0, 2\pi)\}$. The second case uses 2782 random samples in $\Omega$ with 2782 boundary nodes in $\Omega_b$. Figure \ref{Fig.Poisson2} shows our computed solution in this irregular domain using the uniform samples (in (a)) and the random samples (in (b)). The $L_\infty$ error of these solutions is around $3\times 10^{-4}$ and $2\times 10^{-3}$ for the uniform and the random samplings, respectively.

\begin{figure}[!ht]
\centering
\includegraphics[width=\textwidth ]{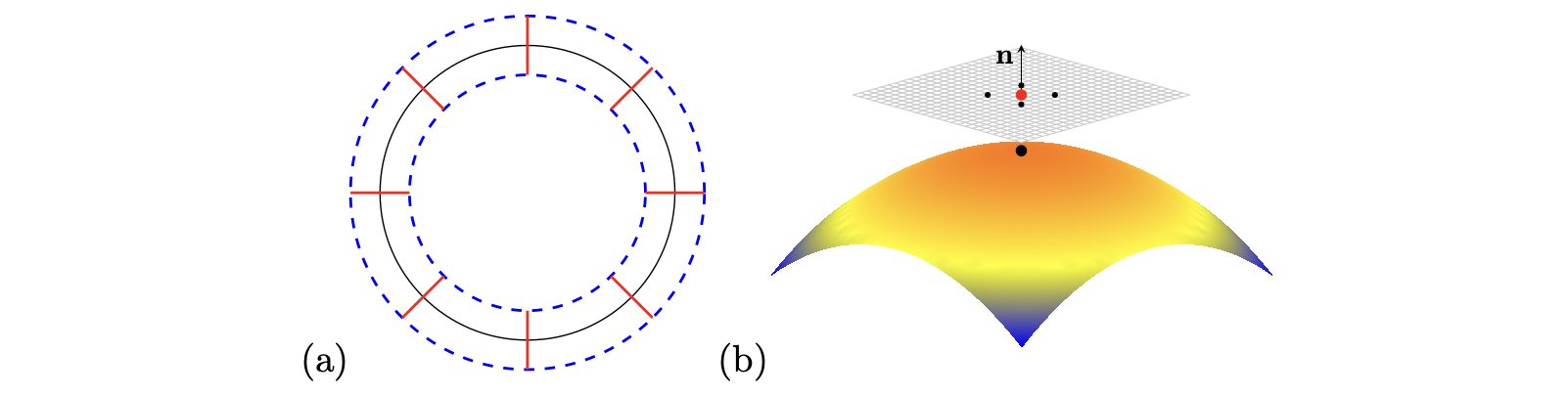}
\caption{(Example \ref{Ex:UnitSphere}) (a) The Eulerian approach to the Laplace-Beltrami operator. The interface is plotted in a black-solid line. Between the two blue dashed lines is the tube area extended from the interface. The red lines show the normal direction of some points on the interface. Along the normal direction, we impose $\n \cdot \nabla f= 0$. (b) The Lagrangian approach to the Laplace-Beltrami operator. The local reconstruction involves the construction of the tangent plane using a normal $\n$ estimated from the data and a projection step of sampled points on a local neighborhood.}
\label{Fig:EulerianLagrangianLB}
\end{figure}

\subsection{The Laplace-Beltrami Operator}
\label{Ex:UnitSphere}

There are several methods for approximating the Laplace-Beltrami (LB) operator on manifolds. One Eulerian approach involves extending the $\mathbb{R}^d$ manifold to $\mathbb{R}^{d+1}$ with the constraint $\mathbf{n} \cdot \nabla f = 0$, where $\mathbf{n}$ denotes the normal on the manifold $S$, as illustrated in Figure \ref{Fig:EulerianLagrangianLB}(a). Various methods have been developed within this Eulerian framework. For instance, an earlier approach by Green \cite{gre05} utilized the level set method \cite{oshset88} and the projection method. The closest point method (CPM) \cite{macruu08,ruumer08,macbraruu11} has also been applied to this problem. Following the basic embedding concept outlined in \cite{wonleu16} for surface eikonal equations, we have devised a similar straightforward embedding approach for the LB eigenproblem, as detailed in \cite{leeleu23}. Additionally, methods like RBF-OGr \cite{pir12} and RBF-FD are utilized within the RBF community. Another approach is the Lagrangian framework, which explicitly deals with sample points on the manifold. In this approach, neighbor points for each sample point are projected onto the local tangent plane of the manifold. Subsequently, these methods use the standard Laplacian on the tangent plane to compute the Laplace-Beltrami operator, incorporating local geometric information. A graphical illustration is provided in Figure \ref{Fig:EulerianLagrangianLB}(b). This technique finds widespread use in point cloud data analysis, as seen in methods proposed in \cite{belsunwan09}, as well as in some local mesh methods developed in \cite{llwz12,lailiazha13}. Other methods, such as the grid-based particle method (GBPM) and the cell-based particle method (CBPM) \cite{leuzha0801,leulowzha11,honleuzha14}, and a virtual grid method \cite{wanleuzha18}, also employ similar principles.

In this example, we focus on approximating the Laplace-Beltrami operator on a two-dimensional surface embedded in $\mathbb{R}^3$ and parameterized by $(s_1,s_2)$. The Laplace-Beltrami operator acting on a function $f:S\rightarrow \mathbb{R}$ is defined as: $\Delta_{S} f = \sum_{\alpha,\beta = 1}^{2} \frac{1}{\sqrt{g}}\frac{\partial}{\partial s_{\alpha}} \left( \sqrt{g} g^{\alpha\beta} \frac{\partial f}{\partial s_{\beta}} \right)$. Here, the metric $[g_{\alpha\beta}]$ is given by $g_{\alpha\beta} = \frac{\partial \mathbf{X}}{\partial s_{\alpha}}\cdot\frac{\partial \mathbf{X}}{\partial s_{\alpha}}$, and the Jacobian of the metric is denoted as $g = g_{11}g_{22}-g_{12}g_{21}$. The matrix $[g^{\alpha\beta}]$ represents the inverse of $[g_{\alpha\beta}]$. It's important to note that the Laplace-Beltrami operator is geometrically intrinsic, meaning it's independent of the parametrization. Thus, computations based on local parametrization remain valid and retain geometric intrinsic properties.

\begin{figure}[!ht]
\centering
\includegraphics[width=\textwidth]{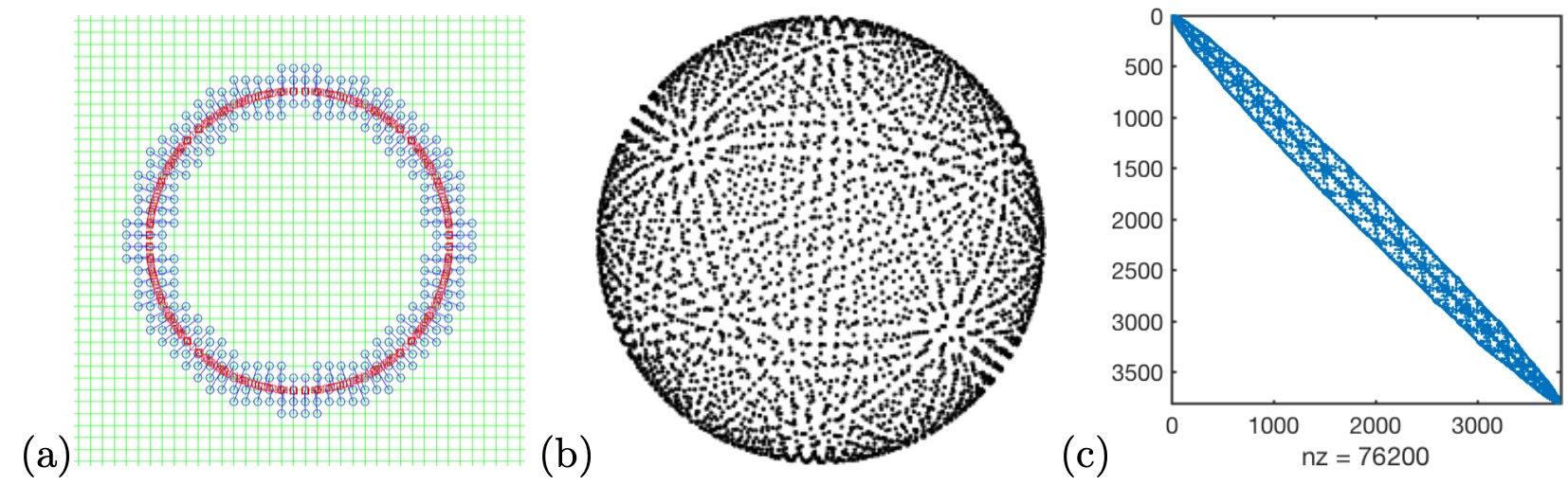}
\caption{(Example \ref{SubSubSec:UnitSphere}) { (a) An illustration of the GBPM/CPM representation of a circle. The underlying mesh is plotted using the green lines. The \textit{activated} grid points are plotted using blue circles, and their projection is shown in red squares. These red squares are the sample points of our manifold. (b) The GBPM sampling of the unit sphere. (c) The sparse pattern of the discretized LB operator after applying the sparse reverse Cuthill-McKee ordering.}}
\label{Fig.GBPMExample}
\end{figure}

\subsubsection{Eigenvalues of the LB Operator}
\label{SubSubSec:UnitSphere}

In this example, we have employed the GBPM or CPM to sample the sphere. This involves selecting points closest to the underlying mesh points near the interface. The illustration in Figure \ref{Fig.GBPMExample}(a) plots a typical point cloud sampling of the interface/surface using GBPM. Here, the underlying mesh is shown with solid lines, while all active grids near the interface (plotted as a circle) are represented by blue circles. The associated closest points on the interface are marked with red squares, and solid line links show the correspondence between each pair. It's worth noting that the resulting point cloud from GBPM may exhibit non-uniformity, as the two closest points to two mesh points can be very close or even identical. A crucial parameter in generating the GBPM/CPM point cloud is the width of the computational band, which we have set to $1.5\Delta x$ in our numerical example. This implies that for each grid point within a distance less than $1.5 \Delta x$ from the interface, we compute its projection onto the manifold. These projected points serve as sample points for representing the interface.

\begin{figure}[!ht]
\centering
\includegraphics[width=\textwidth]{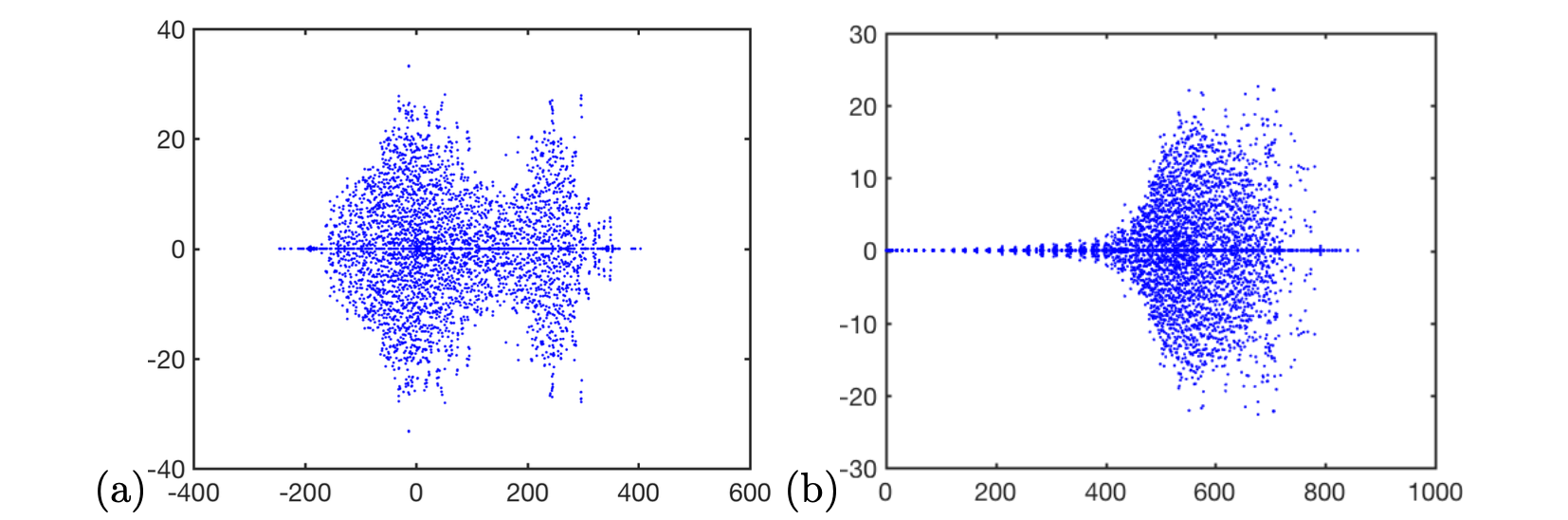}
\caption{(Example \ref{SubSubSec:UnitSphere}) Eigenvalues of the discretized Laplace-Beltrami operator $-\Delta_S$ using different approaches in the complex plane. There are, in total, $N=3795$ nodes in the sample. (a) The least-squares method uses a second-order polynomial in \cite{leulowzha11}. (b) The proposed CLS-GSP method.}
\label{Fig.UnitEigen}
\end{figure}

Here, we compute the eigenvalues of the Laplace-Beltrami operator on the unit sphere represented by the GBPM sampling. The sample points are shown in Figure \ref{Fig.GBPMExample}(b).

First, we construct the DM of the operator $-\Delta_S$ using our proposed CLS-GSP method, denoted as $L$. Then, we compute all the eigenvalues $\lambda$ and eigenvectors $\nu$ of the matrix $L$ using the \textsf{MATLAB} function \textsf{eig}. These eigenvalues and eigenvectors can serve as a good discrete approximation to the eigenvalues and eigenfunctions of the continuous eigenvalue problem $-\Delta_S \nu=\lambda \nu$. After obtaining these computed eigenvalues, we can plot them onto the complex plane. Figure \ref{Fig.UnitEigen} shows the eigenvalues of $L$ from (a) the least-squares method as in the GBPM \cite{leulowzha11}, and (b) our proposed CLS-GSP method.

In many applications, having the real part of all eigenvalues of the Laplace-Beltrami operator $-\Delta_{S}$ staying non-negative is important. For example, when solving the diffusion equation on the manifold $\frac{\partial u}{\partial t}=\Delta_S u$, an eigenvalue of the Laplace-Beltrami operator $-\Delta_{S}$ with a negative real part can lead to instability in the numerical solution. In the GBPM sampling, two sample points can collide at the same place on the surface, leading to their corresponding projection points being arbitrarily close to the tangent plane. However, our proposed CLS-GSP works better in this example. All eigenvalues from our proposed CLS-GSP are on the right-hand side of the complex plane, showing that the real part of all eigenvalues is non-negative.

\begin{figure}[!ht]
\centering
\includegraphics[width=\textwidth]{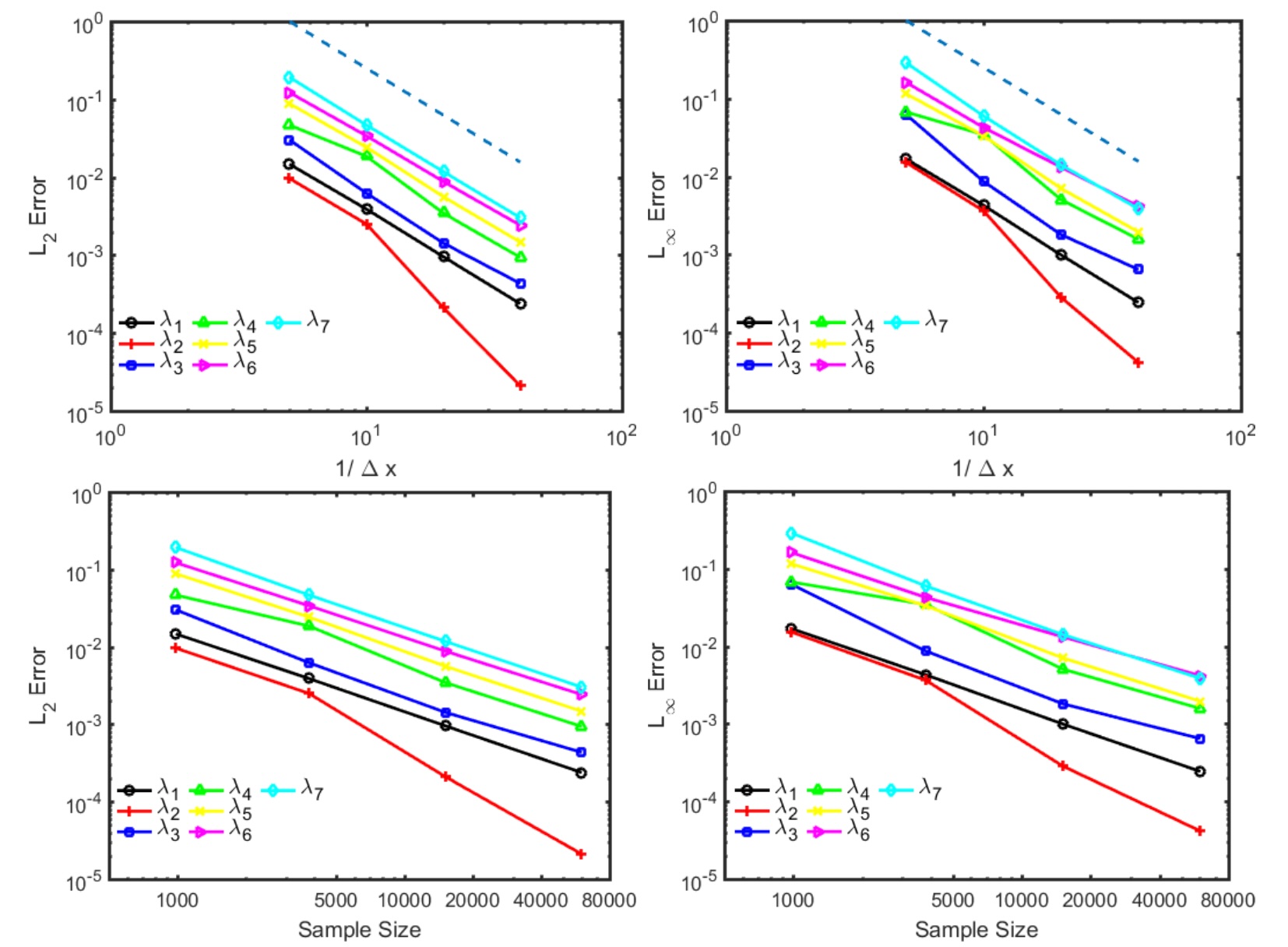}
\caption{(Example \ref{SubSubSec:ConvergenceEigenvaluesLB}) First row: The $E_{2,m}$ and $E_{\infty,m}$ of the first seven eigenvalues in different discretizations. The dashed lines are the reference curve for the second-order convergence. Second row: The $E_{2,m}$ and $E_{\infty,m}$ of the first seven eigenvalues in different sample sizes.}
\label{Fig.EigenConvergence}
\end{figure}

\begin{table}[!ht]
\centering
\begin{tabular}{|c| c |c |c |c|} 
\hline
$\Delta x$ & $0.2$ & $0.1$ & $0.05$ & $0.025$ \\ 
Samples & $984$ & $3795$ & $15127$ & $60322$ \\ 
\hline
$\lambda_1 = 2$&$1.47\times 10^{-2}$&$3.89\times 10^{-3}$&$9.53\times 10^{-4}$&$2.38\times 10^{-4}$\\
$\lambda_2 = 6$&$9.64\times 10^{-3}$&$2.50\times 10^{-3}$&$2.10\times 10^{-4}$&$2.10\times 10^{-5}$\\
$\lambda_3 = 12$&$3.01\times 10^{-2}$&$6.23\times 10^{-3}$&$1.42\times 10^{-3}$&$4.33\times 10^{-4}$\\
$\lambda_4 = 20$&$4.66\times 10^{-2}$&$1.86\times 10^{-2}$&$3.45\times 10^{-3}$&$9.35\times 10^{-4}$\\
$\lambda_5 = 30$&$8.80\times 10^{-2}$&$2.42\times 10^{-2}$&$5.62\times 10^{-3}$&$1.47\times 10^{-3}$\\
$\lambda_6 = 42$&$1.22\times 10^{-1}$&$3.39\times 10^{-2}$&$8.80\times 10^{-3}$&$2.42\times 10^{-3}$\\
$\lambda_7 = 56$&$1.92\times 10^{-1}$&$4.67\times 10^{-2}$&$1.18\times 10^{-2}$&$3.02\times 10^{-3}$\\
\hline
\end{tabular}
\caption{(Example \ref{SubSubSec:ConvergenceEigenvaluesLB}) The $E_{2,m}$ of various eigenvalues of the Laplace-Beltrami operator on the unit sphere sampled by the GBPM with different scales.}
\label{TB.Err2}
\end{table}
\begin{table}[!ht]
\centering
\begin{tabular}{|c| c |c |c |c|} 
\hline
$\Delta x$ & $0.2$ & $0.1$ & $0.05$ & $0.025$ \\
Samples & $984$ & $3795$ & $15127$ & $60322$ \\ 
\hline
$\lambda_1 = 2$&$1.68\times 10^{-2}$&$4.29\times 10^{-3}$&$9.92\times 10^{-4}$&$2.43\times 10^{-4}$\\
$\lambda_2 = 6$&$1.52\times 10^{-2}$&$3.64\times 10^{-3}$&$2.84\times 10^{-4}$&$4.15\times 10^{-5}$\\
$\lambda_3 = 12$&$6.25\times 10^{-2}$&$8.68\times 10^{-3}$&$1.81\times 10^{-3}$&$6.44\times 10^{-4}$\\
$\lambda_4 = 20$&$6.72\times 10^{-2}$&$3.45\times 10^{-2}$&$5.07\times 10^{-3}$&$1.57\times 10^{-3}$\\
$\lambda_5 = 30$&$1.15\times 10^{-1}$&$3.33\times 10^{-2}$&$7.05\times 10^{-3}$&$1.95\times 10^{-3}$\\
$\lambda_6 = 42$&$1.61\times 10^{-1}$&$4.26\times 10^{-2}$&$1.32\times 10^{-3}$&$4.16\times 10^{-3}$\\
$\lambda_7 = 56$&$2.88\times 10^{-1}$&$5.99\times 10^{-2}$&$1.41\times 10^{-2}$&$3.86\times 10^{-3}$\\
\hline
\end{tabular}
\caption{(Example \ref{SubSubSec:ConvergenceEigenvaluesLB}) The $E_{\infty,m}$ of various eigenvalues of the Laplace-Beltrami operator on the unit sphere sampled by the GBPM with different scales.}
\label{TB.ErrInf}
\end{table}

\subsubsection{Convergence of the Eigenvalues of the Laplace-Beltrami Operator}
\label{SubSubSec:ConvergenceEigenvaluesLB}

In this part, we test the convergence of the computed eigenvalues of the LB operator on the unit sphere. On the unit sphere, the exact value of the $m$-th eigenvalue can be analytically computed and is given by $\lambda_m = m(m+1)$ with multiplicity $2m+1$. We use the following normalized $L_2$ and $L_{\infty}$ errors to measure our algorithm's performance:
$
E_{2,m} = \sqrt{\frac{1}{2m+1}\sum_{i}\left(\frac{\lambda_{m,i}-\lambda_m}{\lambda_m} \right)^2}$
and
$E_{\infty,m} = \max_i \left\vert \frac{\lambda_{m,i}-\lambda_m}{\lambda_m} \right\vert$.
Figure \ref{Fig.EigenConvergence} shows that the convergence rate of our algorithm is approximately second order in both error measures. It also shows that the error decreases when the GBPM sample size is increased. Tables \ref{TB.Err2} and \ref{TB.ErrInf} show the details of the errors measured in different discretizations.

\begin{figure}[!ht]
\centering
\includegraphics[width=\textwidth]{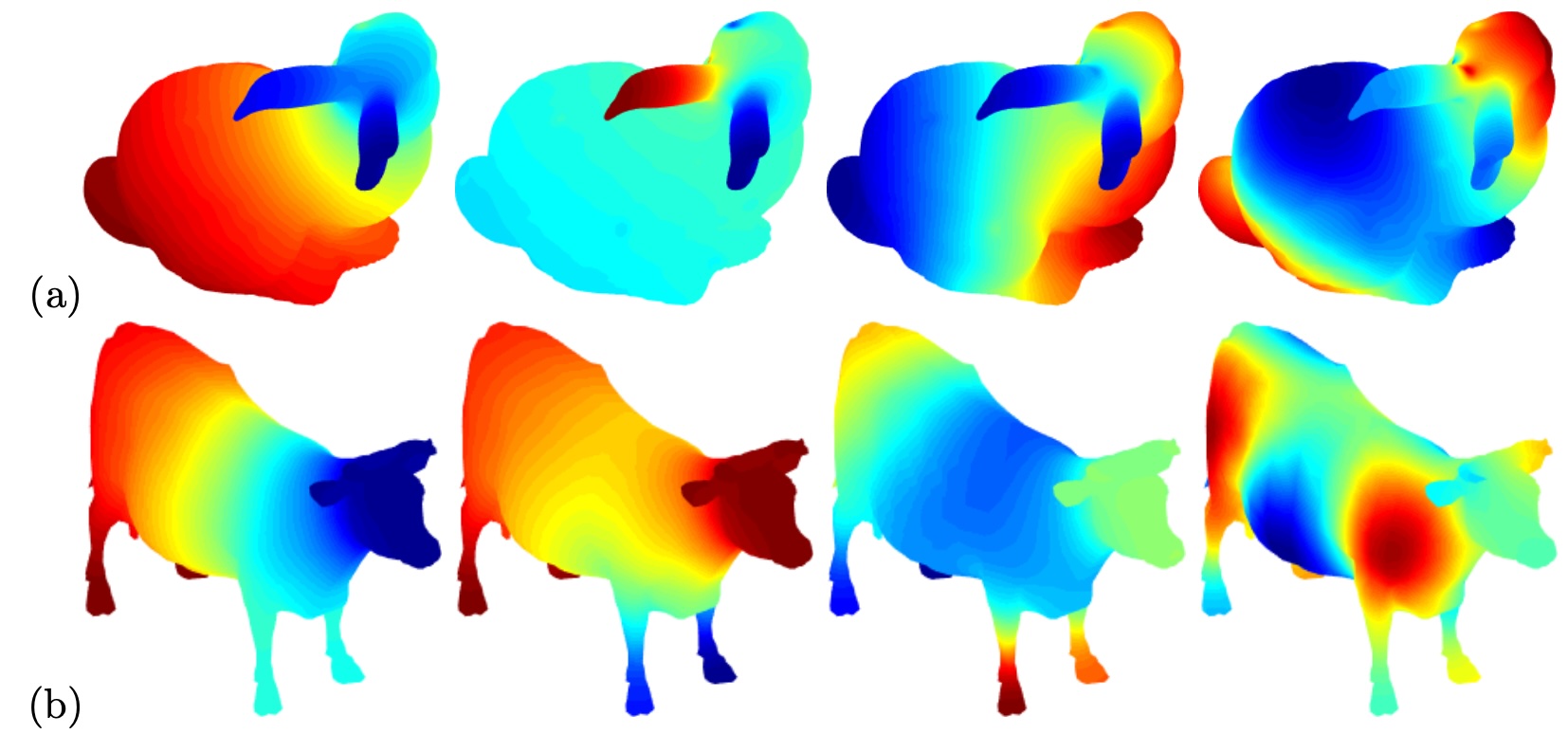}
\caption{(Example \ref{Ex:Manifold}) (a) The 1st, 2nd, 3rd and 6th eigenfunctions of the Stanford bunny. (b) The 1st, 2nd, 6th and 16th eigenfunctions of the cow.}
\label{Fig.Manifold}
\end{figure}

\subsubsection{Eigenfunctions of the LB Operator on Manifolds}
\label{Ex:Manifold}

Our method's ability in handling manifolds, as represented by various samplings, is a significant advantage. Figure \ref{Fig.Manifold} illustrates the first few eigenfunctions of two-point cloud datasets obtained from the Stanford 3D scanning repository and the public domain, showcasing the applicability of your approach across different datasets.

\section{Conclusion}
\label{Summary}

This paper introduced the constrained least-squares ghost sample points (CLS-GSP) method for approximating differential operators on domains sampled by a random set of sample points. Unlike traditional RBF interpolation methods, we formulated the problem as a least-squares one, offering more flexibility in allocating basis locations, especially when sample points are close or overlap. Additionally, we enhanced the approach presented in previous work by enforcing the local reconstruction to pass through the center data point, improving the accuracy and diagonal dominance of the resulting DM.

Given the challenge of designing optimal sample locations, we addressed the issue by introducing the concept of consistency and analyzing the behavior of the pointwise consistency error as sample locations approach the center. By incorporating polynomial basis functions up to the $k$-th order into the CLS-GSP basis, we demonstrated the method's consistency in estimating the Laplacian of a function with an order of $k-1$.

Numerical experiments showcased the accuracy and stability of the proposed CLS-GSP method. Notably, the discretized Laplace-Beltrami operator obtained using CLS-GSP has eigenvalues with positive real parts, indicating stability for solving the diffusion equation on a manifold. Our approach presents a promising framework for robust numerical algorithms in manifold-based applications.

\section*{Acknowledgment}
The work of Leung was supported in part by the Hong Kong RGC grant 16302223.

\bibliographystyle{plain}
\bibliography{syleung}

\end{document}